\documentclass[a4paper,twoside]{article}
\usepackage[utf8]{inputenc}
\usepackage{esvect}
\usepackage{etoolbox}
\usepackage{mathrsfs}
\usepackage{amsmath}
\usepackage[hidelinks]{hyperref}
\usepackage{enumitem}
\usepackage{amsthm}
\usepackage{amssymb}
\usepackage{amsfonts}
\usepackage{graphicx}
\usepackage{hyphenat}
\makeatletter

\def\Str{\mathop{\mathrm{Str}}\nolimits}
\def\dom{\mathop{\mathrm{Dom}}\nolimits}
\def\powerset#1{\mathscr{P}(#1)}
\def\range{\mathop{\mathrm{Range}}\nolimits}
\def\Aut{\mathop{\mathrm{Aut}}\nolimits}
\def\cl{\mathop{\mathrm{Cl}}\nolimits}
\def\Scc{\mathop{\mathrm{Cc}}\nolimits}

\def\str#1{\mathbf {#1}}
\def\ostr#1{{\vv{\mathbf {#1}}}}
\def\vvrev#1{\vv{#1}^{\mathrm{op}}}
\def\arity#1{\mathrm{arity}(\rel{}{#1})}
\def\arityf#1{\mathrm{arity}(\func{}{#1})}
\def\nbfunc#1#2{F\ifstrempty{#1}{}{_{#1}}\ifstrempty{#2}{}{^{#2}}}
\def\nbrel#1#2{R\ifstrempty{#1}{}{_{#1}}\ifstrempty{#2}{}{^{#2}}}
\def\rel#1#2{\nbrel{\ifstrempty{#1}{}{\str{#1}}}{#2}}
\def\func#1#2{\nbfunc{\ifstrempty{#1}{}{\str{#1}}}{#2}}
\def\relfunc#1#2{R_{\mathbf{#1}^-}^{F#2}}
\def\F{{\mathcal F}}
\def\K{{\mathcal K}}
\def\Fraisse{Fra\"{\i}ss\' e}

\theoremstyle{definition}
\newtheorem{defn}{Definition}[section]
\newtheorem*{example}{Example}
\newtheorem*{remark}{Remark}
\theoremstyle{remark}

\theoremstyle{plain}
\newtheorem{thm}{Theorem}[section]
\newtheorem{corollary}[thm]{Corollary}
\newtheorem{prop}[thm]{Proposition}
\newtheorem{lem}[thm]{Lemma}

\begin{document}
\bibliographystyle{plain}

\title{Ramsey properties and extending partial automorphisms for classes of finite structures}

\author{David M. Evans\thanks{Department of Mathematics, Imperial College London, London SW7~2AZ, UK, \texttt{david.evans@imperial.ac.uk}},
	Jan Hubi\v cka\thanks{Supported  by  project  18-13685Y  of  the  Czech  Science Foundation (GA\v CR), Center for Foundations of Modern Computer Science (Charles University project UNCE/SCI/004), by the PRIMUS/17/SCI/3 project of Charles University and by ERC Synergy grant DYNASNET. Department of Applied Mathematics (KAM), Charles University, Malostransk\' e n\' am. 25, 11800 Praha, Czech Republic, 118 00 Praha 1, \texttt{hubicka@kam.mff.cuni.cz}},
        Jaroslav Ne\v set\v ril\thanks{Supported  ERC Synergy grant DYNASNET and by European Associated Laboratory (LEA STRUCO). Computer Science Institute of Charles University (IUUK), Charles University, Malostransk\' e n\' am. 25, 11800 Praha, Czech Republic, 118 00 Praha 1, \texttt{nesetril@iuuk.mff.cuni.cz}}}
\date{}
\maketitle

\begin{abstract}
We show that every free amalgamation class of finite structures with
relations and (set-valued) functions is a Ramsey class when enriched by a free linear ordering of vertices. This is a common strengthening of the Ne\v set\v ril-R\"odl Theorem and the second and third
authors' Ramsey theorem for finite models (that is, structures with both relations and
functions).
We also find subclasses with the ordering property.
For languages with relational symbols and unary functions we also show the extension property for partial automorphisms (EPPA)
of free amalgamation classes. These general results solve several conjectures and provide an easy Ramseyness test for many classes of structures.
\end{abstract}

\section{Introduction}
In this paper we discuss three related concepts ---  Ramsey classes, the
ordering (or lift or expansion) property and the extension property for partial
automorphisms (EPPA). The main novelty of our results is that they hold for  free amalgamation classes of finite structures
with both relations and \emph{set-valued functions}. This provides a useful tool for proving these types of results for some classes of structures which naturally carry a closure operation. We will explain below what we mean by this; examples and applications are given at the end of the paper.

As is well known, all three of these concepts about classes of finite structures are related to issues in topological dynamics and this relationship provides much of the motivation for what we do.  For example, by~\cite{Kechris2005} the automorphism group of the \Fraisse{} limit 
of a Ramsey class $\mathcal R$ is extremely amenable. Moreover, if the Ramsey class
$\mathcal R$ has the ordering property with respect to some amalgamation class $\mathcal K$, then it
determines the universal minimal flow of the automorphism group of the \Fraisse{} limit of ${\mathcal K}$. Thus, our results Theorem~\ref{thm:main} and Theorem~\ref{thm:ordering} about the Ramsey and ordering properties
give new examples of this correspondence.
  By~\cite{Kechris2007} and our 
Theorem~\ref{EPPA} about EPPA, the automorphism group of the \Fraisse{} limit of every free
amalgamation class $\K$ in a language where all functions are unary is amenable. However we note that the same conclusion, without the restriction that the (set-valued) functions are unary, also follows from our Ramsey Theorem~\ref{thm:main} and Proposition 9.3 of \cite{AKL14} (note that the assumption of the ordering property in the statement of Proposition 9.3 of \cite{AKL14} is not necessary, as was observed by Pawliuk and Soki{\'c} in~\cite{PawliukSokic16}). 

To generalise naturally the known results
about relational structures, we need to define carefully what we mean by a structure and  substructure,
because none of these results holds in the context of free amalgamation
classes with strong embeddings (as discussed in~\cite{Evans2}, see also the remarks preceding Theorem~\ref{thm:main}). 
In the Ramsey theory setting it is common to work with `incomplete' structures.  Thus we have to modify the standard model-theoretic notion of structures (see e.g.~\cite{Hodges1993}), where
functions are required to be total and thus complete in some sense.
 Before stating the main results, we give the basic (model-theoretic) setting of this paper.
We find it useful to introduce a variant
of the usual model-theoretic structures, allowing set-valued functions, which is well tailored both to the Ramsey and EPPA setting.

Let $L=L_{\mathcal R}\cup L_{\mathcal F}$ be a \emph{language} with \emph{relational symbols} $\rel{}{}\in L_{\mathcal R}$ and \emph{function symbols} $F\in L_{\mathcal F}$ each having its \emph{arity} denoted by $\arity{}$ for relations and $\arityf{}$ for functions.
Denote by $\powerset{A}$ the set of all subsets of $A$.
An \emph{$L$-structure} $\str{A}$ is a structure with {\em vertex set} $A$, functions $\func{A}{}\colon A^{\arityf{}}\to \powerset{A}$, $\func{}{}\in L_\mathcal F$ and relations $\rel{A}{}\subseteq A^{\arity{}}$, $\rel{}{}\in L_\mathcal R$.
Set-valued functions permits an explicit description of algebraic closures in the \Fraisse{} limits without changing the automorphism group which is necessary for some applications discussed in Section~\ref{sec:examples}.
Because the image of a tuple can be empty set it also gives a natural meaning to the notion of free amalgamation for structures in languages containing functions of arity greater or equal to 2 and it simplifies some of the notation below.

The language is usually fixed and understood from the context (and it is in most cases denoted by $L$).  If the set $A$ is finite we call $\str A$ a\emph{ finite structure}. We consider only structures with finitely or countably infinitely many vertices. 
If the language $L$ contains no function symbols, we call $L$ a {\em relational language} and say that an  $L$-structure is  a {\em relational $L$-structure}.
A function symbol $\func{}{}$ such that $\arityf{}=1$ is a {\em unary function}.

The notions of embedding, isomorphism, homomorphisms and free amalgamation classes
are natural generalisations of the corresponding notions on relational structures
and are formally introduced in Section~\ref{sec:background}. Considering function
symbols has important consequences for what we consider a substructure.
  An $L$-structure $\str{A}$  is a {\em substructure} of $\str{B}$ if $A\subseteq B$ and all relations and functions of $\str{B}$ restricted to $A$
are precisely relations and functions of $A$.
 In particular for every $F\in L$ and every tuple $(v_1,v_2,\ldots, v_{\arityf{}})$ of vertices
of $A$ it also holds that $\func{B}{}(v_1,v_2,\ldots, v_{\arityf{}})\subseteq A$.
This implies that $\str{B}$ does not induce a substructure on every subset of $\str{B}$ (but only on `closed' sets, to be defined later).

Building on these model theoretic notions we now outline the contents of this paper.
We proceed to the three main directions---Ramsey theory, the ordering property and the extension property for partial automorphisms (EPPA).
In each of these directions we now state the main result (stated below as Theorems~\ref{thm:main}, \ref{thm:ordering} and \ref{EPPA}). This can be summarised by saying that for free amalgamation classes we have strong positive theorems in each of these areas. There is more to it than just meets the eye: for the first time we demonstrate affinity of all these directions (for the ordering property and Ramsey this is known, but for EPPA much less so).

This has a number of applications to special classes of structures. In Section~\ref{sec:examples} we give several examples which have received recent attention: $k$-orientations, bowtie-free graphs and Steiner systems. These all are easy consequences of our main result. Finally, let us remark that our results further narrow the gap for the project of characterisation of Ramsey classes~\cite{Nevsetril2005}.
\subsection{Ramsey classes}

For structures $\str{A},\str{B}$ denote by ${\str{B}\choose \str{A}}$ the set of all sub-structures of $\str{B}$, which are isomorphic to $\str{A}$.  Using this notation the definition of a Ramsey class has the following form.
A class $\mathcal C$ is a \emph{Ramsey class} if for every two objects $\str{A}$ and $\str{B}$ in $\mathcal C$ and for every positive integer $k$ there exists a structure $\str{C}$ in $\mathcal C$ such that the following holds: For every partition ${\str{C}\choose \str{A}}$ into $k$ classes there exists a $\widetilde{\str B} \in {\str{C}\choose \str{B}}$ such that ${\widetilde{\str{B}}\choose \str{A}}$ belongs to one class of the partition.  It is usual to shorten the last part of the definition to $\str{C} \longrightarrow (\str{B})^{\str{A}}_k$.

We are motivated by the following, now classical, result.
\begin{thm}[Ne\v set\v ril-R\"odl Theorem~\cite{Nevsetvril1977,Nevsetvril1977b}]
\label{thm:NRoriginal}
Let $L$ be a relational language, $(\str{A},\leq)$ and $(\str{B},\leq)$ be ordered $L$-structures. Then there exists an ordered $L$-structure
$(\str{C},\leq)$ such that $(\str{C},\leq)\longrightarrow (\str{B},\leq)^{(\str{A},\leq)}_2$.

Moreover, if $\str{F}$ is an irreducible $L$-structure  (see Section~\ref{sec:Ramsey} for the definition of irreducibility) and $\str{A}$ and $\str{B}$ do not contain $\str{F}$ as a substructure, then $\str{C}$ may be chosen with the same property.
\end{thm}
In our setting this result may be reformulated as follows:
Given a language $L$, denote by $\vv{L}$ the  language $L$ extended by one
binary relation $\leq$.  Given an $L$-structure $\str{A}$, an {\em ordering of $\str{A}$}
is an $\vv{L}$-structure extending $\str{A}$
by an arbitrary linear ordering $\leq_\str{A}$ of the vertices. 
For brevity we denote such ordered $\str{A}$ as $\vv{\str{A}}$.
Given a class $\K$ of $L$-structures, denote by $\vv{\K}$ the class of
all orderings of structures in $\K$.
We sometimes say that $\vv{\K}$ arises by taking {\em free orderings} of structures in $\mathcal K$.
Note that minimal relational structures which do not belong to a free amalgamation class are all irreducible. Thus
Theorem~\ref{thm:NRoriginal} can now be re-formulated using basic notions of \Fraisse{} theory (which will be introduced in Section~\ref{sec:background})
as follows:
\begin{thm}[Ne\v set\v ril-R\"odl Theorem for free amalgamation classes]
\label{thm:NR}
Let $L$ be a relational language and $\K$ be a free amalgamation class of relational $L$-structures.  Then $\vv{\K}$ is a Ramsey class.
\end{thm}
The more recent connection between Ramsey classes and extremely amenable
groups~\cite{Kechris2005} has motivated a systematic search for new examples of Ramsey classes. It
became apparent that it is important to consider structures with both relations
and functions or, equivalently, classes of structures with ``strong embeddings''.
  This led to \cite{Hubicka2016} which provides a sufficient structural
condition for a subclass of a Ramsey class of structures to be Ramsey and also 
generalises this approach to classes of structures (containing relations and functions) with formally-described closures. 
It is however clear
that considering classes with closures (especially non-unary closures) leads to many technical difficulties. In fact,
a recent example given by first author based on Hrushovski's predimension
construction~\cite{Evans2} not only answers one of the main questions in the area (about the existence of precompact Ramsey expansions), but also shows that there is no  direct generalisation of the Ne\v set\v ril-R\"odl Theorem
to a free amalgamation classes with strong embeddings.
However, perhaps surprisingly, we show that if closures can be explicitly represented by means
of set-valued functions such that strong embeddings are just embeddings, such a statement is true. We prove:
\begin{thm}
\label{thm:main}
Let $L$ be a language (involving relational symbols and set-valued functions) and let 
$\K$ be a free amalgamation class of $L$-structures. 
 Then $\vv{\K}$ is a Ramsey class.
\end{thm}
This yields an alternative proof of the Ramsey property for some recently-discovered Ramsey classes (such as ordered partial Steiner systems~\cite{bhat2016ramsey}, bowtie-free graphs~\cite{Hubivcka2014}, bouquet-free graphs~\cite{Cherlin2007}) and also for new classes: most importantly a Ramsey expansion
of the class of $2$-orientations of a Hrushovski predimension construction which is elaborated in~\cite{Evans2} and which was one of the main
motivations for this paper.

\subsection{Ordering property}
\label{sec:orderingpropintro}

A class $\mathcal O\subseteq \vv{\K}$ has the {\em ordering property (with respect to $\K$)} if
for every $\str{A}\in \K$ there exists $\str{B}\in \K$ such that every ordering $\vv{B}\in \mathcal O$
of $\str{B}$ contains a copy of every ordering $\vv{A}\in \mathcal O$ of $\str{A}$.
It is well known that for every free amalgamation class $\K$ of relational
structures the class $\vv{\K}$ has ordering property.  This fact
follows by an application of Theorem~\ref{thm:NR} but can also be shown by
more general methods based on hypergraphs of large girth~\cite{Nesetril1975,Nesetvril1978}.
This shows that there are many classes $\K$ of relational structures for which
$\vv{\K}$ has the ordering property (with respect to $\K$) but $\K$ itself is not a Ramsey class.

For languages containing function symbols, the situation is more complicated. To see that some extra restriction on our class $\mathcal O$ is required, we note the following example. Denote by $\mathcal T$ the class of all
finite forests represented by a single unary function $\func{}{}$ connecting a vertex to its
father. Let $\str{A}$ be  a structure containing two vertices $a,b$ and $\func{A}{}(a)=\{b\}$.
A vertex $c$ is a {\em root} if $\func{A}{}(c)=\emptyset$.  Any structure $\str{B}$ can
be ordered in increasing order according to the distance from a root vertex.  It follows
that such an ordering never contains the ordering of $\str{A}$ given by $a\leq_\str{A} b$
and consequently $\vv{\mathcal T}$ does not have the ordering property.

 Nevertheless, we show the following:
\begin{thm}
\label{thm:ordering}
Let $L$ be a language (involving relational symbols and set-valued functions) and let $\K$ be a free amalgamation class of $L$-structures. 
Then there exists an explicitly-defined amalgamation class $\mathcal O \subseteq \vv{\K}$
of {\em admissible orderings} such that:
\begin{enumerate}
 \item every $\str{A}\in \K$ has an ordering in $\mathcal O$;
 \item $\mathcal O$ is a Ramsey class; and,
 \item $\mathcal O$ has the ordering property (with respect to $\mathcal K$).
\end{enumerate}
\end{thm}

The details of the admissible orderings are technical and are described in full in Definition~\ref{def:admissible}. The existence of a subclass $\mathcal O \subseteq \vv{\K}$ with the above three properties follows directly from Theorem 10.7 of~\cite{Kechris2005} and Theorem~\ref{thm:main}. The novelty in Theorem~\ref{thm:ordering} is the explicit description of a class $\mathcal O$. The proof of Theorem~\ref{thm:ordering} is a combination of the Ramsey methods used to show the ordering property of classes of
relational structures and the methods used to show the ordering property of classes with unary
functions.
\subsection{Extension property for partial automorphisms -- EPPA}
A {\em partial automorphism} of an $L$-structure $\str{A}$ is an isomorphism $f : \str{D}
\to \str{E}$ of substructures $\str{D}, \str{E}$ of $\str{A}$.  We
say that a class of finite $L$-structures $\K$  has the {\em extension property for
partial automorphisms}  ({\em EPPA}, sometimes called the {\em Hrushovski extension
property}) if whenever $\str{A} \in \K$ there is $\str{B} \in \K$ such that
$\str{A}$ is substructure of $\str{B}$ and every partial automorphism
of $\str{A}$ extends to an automorphism of $\str{B}$, see~\cite{hrushovski1992,Herwig1995,herwig2000,hodkinson2003,Solecki2005,vershik2008}.  In the following we will
simply call $\str{B}$  with the property above an {\em EPPA-witness} of
$\str{A}$.

For relational languages, the extension property for partial automorphisms of free amalgamation classes
can be derived from the following strengthening of the extension property for partial automorphisms:
\begin{thm}[Hodkinson-Otto \cite{hodkinson2003}]
\label{thm:otto-hodkinson}
Let $L$ be a relational language, then for every finite  $L$-structure $\str{A}$
there exists a finite and clique faithful EPPA-witness $\str{B}$.
\end{thm}

A {\em clique faithful EPPA-witness}  $\str{B}$ is an EPPA-witness of $\str{A}$ with the additional
property that for every clique $\str{C}$ in the Gaifman graph of $\str{B}$ there exists an 
automorphism $g$ of $\str{B}$ such that $g(C)\subseteq A$. 
(For a relational $L$-structure $\str{A}$ the \emph{Gaifman graph}  is the graph $\str{G}_\str{A}$ with
vertices $A$ and  edges those pairs of vertices contained in a tuple of a relation of
$\str{A}$: $\str{G}_\str{A}=(V,E)$ where $\{x,y\}\in E$ if and only
if $x\neq y$ and there exists tuple $\vv{t}\in \rel{A}{},\rel{}{}\in L$ such that
$x,y\in \vv{t}$.)

It is a well known fact
that free amalgamation classes can be equivalently described by forbidden embeddings
from a family of structures whose Gaifman graph is a clique and consequently Theorem~\ref{thm:otto-hodkinson}
implies that every free amalgamation class of relational structures has EPPA (see, for example, \cite{Siniora}).

The notion of {\em irreducibility}  of a structure (given in Definition~\ref{def:irreducible}) is a natural generalisation to the context of functional languages of
the above notion of a clique in a graph.
We say that an EPPA-witness $\str{B}$ of $\str{A}$ is {\em irreducible substructure faithful} if
for every irreducible substructure $\str{C}$ of $\str{B}$ there exists an 
automorphism $g$ of $\str{B}$ such that $g(C)\subseteq A$.

Theorem~\ref{thm:otto-hodkinson} was further strengthened by Siniora and Solecki in
the following form.

\begin{thm}[Siniora-Solecki \cite{Siniora}]
\label{thm:siniora-solecki}
Let $L$ be relational language. Then for every finite relational $L$-structure $\str{A}$
there exists a finite clique faithful and coherent EPPA-witness $\str{B}$. 
\end{thm}

Let $X$ be a set and $\mathcal P$ be a family of partial bijections between subsets
of $X$. A triple $(f, g, h)$ from $\mathcal P$ is called a {\em coherent triple} if $\dom(f) = \dom(h), \range(f ) = \dom(g), \range(g) = \range(h)$ and $h = g \circ f$.

Let $X$ and $Y$ be sets, and $\mathcal P$ and $\mathcal Q$ be families of partial bijections between subsets
of $X$ and between subsets of $Y$, respectively. A function $\varphi: \mathcal P \to \mathcal Q$ is said to be a
{\em coherent map} if for each coherent triple $(f, g, h)$ from $\mathcal P$, its image $\varphi(f), \varphi(g), \varphi(h)$ in $\mathcal Q$ is coherent.

An EPPA-witness $\str{B}$ of $\str{A}$ is {\em coherent} if every
partial automorphism $f$ extends to some $\hat{f} \in \Aut(\str{B})$ with  the property that the map $\varphi$ from partial automorphisms of $\str{A}$ to automorphisms of $\str{B}$  given by $\varphi(f) = \hat{f}$ is coherent.

Our third main result is a strengthening of all of the above results to classes of structures
with unary functions. We do not know how to strengthen this theorem for higher arities. 
\begin{thm}
\label{EPPA}
Let $L$ be a language such that every function symbol $F\in L$ is unary. 
Then for every finite $L$-structure $\str{A}$ there exists a finite, irreducible substructure faithful, coherent EPPA-witness $\str{B}$.
Consequently every free amalgamation class $\mathcal K$ of finite $L$-structures has the coherent extension property for partial automorphisms.
\end{thm}

\subsection{Further background and notation}
\label{sec:background}
We now review some standard graph-theoretic and model-theoretic notions (see e.g.~\cite{Hodges1993}).

A \emph{homomorphism} $f:\str{A}\to \str{B}$ is a mapping $f:A\to B$ such that:
\begin{enumerate}
\item[(a)] for every relation symbol  $\rel{}{}\in L_{\mathcal R}$ of arity $a$ it holds:
$$(x_1,x_2,\ldots, x_{a})\in \rel{A}{}\implies (f(x_1),f(x_2),\ldots,f(x_{a}))\allowbreak\in \rel{B}{},$$
\item[(b)] for every function symbol $\func{}{}\in L_{\mathcal F}$ of arity $a$ it holds:
$$f(\func{A}{}(x_1,x_2,\allowbreak \ldots, x_{a}))\subseteq \func{B}{}(f(x_1),f(x_2),\ldots,\allowbreak f(x_{a})).$$
\end{enumerate}
 Here, for a subset $A'\subseteq A$ we denote by $f(A')$ the set $\{f(x): x\in A'\}$ and by $f(\str{A})$ the homomorphic image of a structure $\str{A}$.

 If $f$ is injective, then $f$ is called a \emph{monomorphism}. A monomorphism $f$ is an \emph{embedding} if it holds:
\begin{enumerate}
\item[(a)] for every relation symbol  $\rel{}{}\in L_{\mathcal R}$ of arity $a$ it holds:
$$(x_1,x_2,\ldots, x_{a})\in \rel{A}{}\iff (f(x_1),f(x_2),\ldots,f(x_{a}))\allowbreak\in \rel{B}{},$$
\item[(b)] for every function symbol $\func{}{}\in L_{\mathcal F}$ of arity $a$ it holds:
$$f(\func{A}{}(x_1,x_2,\allowbreak \ldots, x_{a})) = \func{B}{}(f(x_1),f(x_2),\ldots,\allowbreak f(x_{a})).$$
\end{enumerate}
  If $f$ is an embedding which is an inclusion then $\str{A}$ is a \emph{substructure} (or \emph{subobject}) of $\str{B}$. For an embedding $f:\str{A}\to \str{B}$ we say that $\str{A}$ is \emph{isomorphic} to $f(\str{A})$ and $f(\str{A})$ is also called a \emph{copy} of $\str{A}$ in $\str{B}$. Thus $\str{B}\choose \str{A}$ is defined as the set of all copies of $\str{A}$ in $\str{B}$.

Given $\str{A}\in \K$ and $B\subset A$, the {\em closure of $B$ in $\str{A}$}, denoted by $\cl_\str{A}(B)$, is the smallest substructure of $\str{A}$ containing $B$.
Closure in $\str{A}$ is {\em unary} if $\cl_\str{A}(B)=\bigcup_{v\in B}\cl_\str{A}(v)$ for all $B \subset A$.

\begin{figure}
\centering
\includegraphics{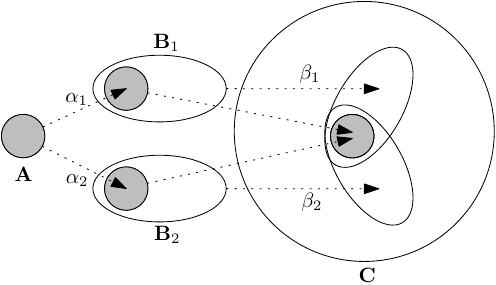}
\caption{An amalgamation of $\str{B}_1$ and $\str{B}_2$ over $\str{A}$.}
\label{amalgamfig}
\end{figure}
Let $\str{A}$, $\str{B}_1$ and $\str{B}_2$ be structures, $\alpha_1$ an embedding of $\str{A}$
into $\str{B}_1$ and $\alpha_2$ an embedding of $\str{A}$ into $\str{B}_2$. Then
every structure $\str{C}$
 with embeddings $\beta_1:\str{B}_1 \to \str{C}$ and
$\beta_2:\str{B}_2\to\str{C}$ such that $\beta_1\circ\alpha_1 =
\beta_2\circ\alpha_2$ is called an \emph{amalgamation} of $\str{B}_1$ and $\str{B}_2$ over $\str{A}$ with respect to $\alpha_1$ and $\alpha_2$. See Figure~\ref{amalgamfig}.
We will call $\str{C}$ simply an \emph{amalgamation} of $\str{B}_1$ and $\str{B}_2$ over $\str{A}$
(as in most cases $\alpha_1$ and $\alpha_2$ can be chosen to be inclusion embeddings).
We say that such an amalgamation is \emph{free} if $\beta_1(x_1)=\beta_2(x_2)$ if and only if $x_1\in \alpha_1(A)$ and $x_2\in \alpha_2(A)$,
$C=\beta_1(B_1)\cup \beta_2(B_2)$ and whenever a tuple $\vv{t}$ of vertices of $\str{C}$ contains vertices of both
$\beta_1(B_1\setminus \alpha_1(A))$ and $\beta_2(B_2\setminus \alpha_2(A))$, then $\vv{t}$ is in no relation of $\str C$,
and for every function $\func{}{}\in L$ with $\arityf{} = |\vv{t}|$ it holds that $\func{C}{}(\vv{t})=\emptyset$.

\begin{defn}
\label{defn:amalg}
An \emph{amalgamation class} is a class $\K$ of finite structures satisfying the following three conditions:
\begin{enumerate}
\item {\em Hereditary property:} For every $\str{A}\in \K$ and a substructure $\str{B}$ of $\str{A}$ we have $\str{B}\in \K$;
\item {\em Joint embedding property:} For every $\str{A}, \str{B}\in \K$ there exists $\str{C}\in \K$ such that $\str{C}$ contains both $\str{A}$ and $\str{B}$ as substructures;
\item {\em Amalgamation property:} 
For $\str{A},\str{B}_1,\str{B}_2\in \K$ and $\alpha_1$ embedding of $\str{A}$ into $\str{B}_1$, $\alpha_2$ embedding of $\str{A}$ into $\str{B}_2$, there is $\str{C}\in \K$ which is an amalgamation of $\str{B}_1$ and $\str{B}_2$ over $\str{A}$ with respect to $\alpha_1$ and $\alpha_2$.
\end{enumerate}
If the $\str{C}$ in the amalgamation property can always be chosen as the free amalgamation, then $\K$ is a {\em free amalgamation class}. 
\end{defn}

We will give examples of free amalgamation classes (in the case where the language $L$ is not relational)  in Section \ref{sec:examples}.

\section{Free amalgamation classes are Ramsey}
\label{sec:Ramsey}
The proof of Theorem~\ref{thm:main} is a variation of the {\em Partite Construction} introduced by  Ne\v set\v
ril and R\"odl for classes of hypergraphs and relational structures (see~\cite{Nevsetvril1989}) which was recently extended 
to classes with unary~\cite{Hubivcka2014} and later general
closures (or functions)~\cite{Hubicka2016}.  The Partite Construction is 
a machinery which allows one to transform one Ramsey class into another, more special,
Ramsey class.  What follows is a Partite Construction proof (as used previously, for example, in \cite{Nevsetvril2007}) done in the context
of structures involving set-valued functions.

The basic Ramsey class in our construction will be provided by the following 
result about the Ramsey property of ordered structures with relations and (total) functions.
\begin{thm}[Ramsey theorem for finite models, Theorem 2.19 of~\cite{Hubicka2016}]
\label{thm:models}
Suppose $L$ is a language containing a relational symbol $\leq$ and let 
$\vv{\Str}(L)$ be the class of all finite $L$-structures where
$\leq$ is a linear ordering of the vertices.  Then $\vv{\Str}(L)$ is a Ramsey class.
\end{thm}
This is a strengthening of the theorem giving the Ramsey property of ordered relational
structures proved independently by Ne\v set\v ril-R\"odl~\cite{Nevsetvril1977b}
and Abramson-Harrington~\cite{Abramson1978} in 1970s. Considering functions is a rather difficult task and the proof of Theorem~\ref{thm:models} involves a recursive
nesting of the Partite Constructions to establish valid non-unary closures.
Building on Theorem~\ref{thm:models} our task is significantly easier and we
only concentrate on further refining the Ramsey structure given by Theorem~\ref{thm:models} into one belonging
to a given free amalgamation class. This is done by tracking all irreducible substructures of the object constructed.

We shall remark that recent paper of second two authors~\cite{Hubicka2016} defines structures with partial function to singletons only.
In the Ramsey context this is however equivalent because the structures are ordered.
Thus it is possible to replace
a function $F$ by symbols $\func{}{1},\func{}{2},\func{}{3},\ldots$ and put
$\func{}{i}(\vv{t})$ to be the $i$-th element of $F(\vv{t})$ whenever it is defined.
With this it may be seen then that Theorem 2.11 of~\cite{Hubicka2016} implies Theorem~\ref{thm:main}. However
the proof presented here is cleaner and much simpler.

 As mentioned before, a relational structure $\str{A}$ is {\em irreducible}
if every pair of vertices belongs to some tuple in a relation of $\str{A}$.
It is well known that every free amalgamation class $\K$ of relational
structures can be equivalently described as a class of finite structures
that contains no copies of structures from 
a fixed family $\F$ of irreducible relational structures (see, for example, \cite{Hubicka2016}).
In fact the family $\F$ consists of all minimal structures not belonging to class $\K$.
This easy observation explains the correspondence of Theorem~\ref{thm:NRoriginal} and Theorem~\ref{thm:NR}.

Our construction is based on the following refinement of the notion of
irreducible structure in the context of structures with functions
(which allows us to strengthen the construction in~\cite{Hubicka2016}):

\begin{defn}
\label{def:irreducible}
An $L$-structure $\str{A}$ is \emph{irreducible} if it cannot be created as a free amalgamation of any two of its proper substructures.
\end{defn}
\begin{example}
Consider the language $L$ consisting of one binary relation $\rel{}{}$ and one unary function $\func{}{}$.
An example of an irreducible structure is a structure $\str{A}$ (depicted in Figure~\ref{irreduciblefig}) on vertices $A=\{a,b,c,d\}$ where $(a,b)\in \rel{A}{}$,
 $\func{A}{}(a)=\{c\}$, $\func{A}{}(b)=\{d\}$ and $\func{A}{}(a)=\func{A}{}(b)=\emptyset$. This structure is reducible if $\func{}{}$ is seen
as a relation rather than a function.
\begin{figure}
\centering
\includegraphics{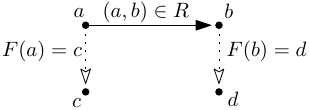}
\caption{An example of an irreducible structure with a binary relation $R$ and a unary function $F$.}
\label{irreduciblefig}
\end{figure}
\end{example}
\label{sec:partlem}
The basic part of our construction of Ramsey objects is a variant of the Partite Lemma~(introduced in \cite{Nevsetvril1989} and refined for closures in \cite{Hubicka2016}) which deals with the following objects.
\begin{defn}[$\str{A}$-partite system]
Let $L$ be a language and $\str{A}$ an  $L$-structure. Assume $A = \{1, 2,\ldots, a\}$.  An \emph{$\str{A}$-partite $L$-system} is a tuple $(\str{A},{\mathcal X}_\str{B},\allowbreak \str{B})$ 
where $\str{B}$ is an $L$-struc\-ture and $\mathcal X_\str{B}=\{X^1_\str{B},X^2_\str{B},\ldots, X^a_\str{B}\}$  is a partition of the vertex set of $\str{B}$ into $a$ many classes $X^i_\str{B}$, called \emph{parts} of $\str{B}$,  such that 
\begin{enumerate}
\item the mapping $\pi$ which maps every $x \in X^i_\str{B}$ to $i$, $i = 1,2,\ldots,a$, is a
homo{\-}morphism
 $\str{B}\to\str{A}$ ($\pi$ is called the \emph{projection}); 
\item every tuple in every relation of $\str{B}$ meets every class $ X^i_\str{B}$ in at most one element (i.e. these tuples are called \emph{transversal} with respect to the partition).
\item for every function $\func{}{}\in L_\mathcal F$ it holds that for every $\vv{t}\in A^{\arityf{}}$ such that $\func{A}{}(\vv{t})\neq \emptyset$ the tuple created by concatenation of $\vv{t}$ and $\func{B}{}(\vv{t})$ (in any order) is transversal.
\end{enumerate}
\end{defn}
The \emph{isomorphisms} and \emph{embeddings} of $\str{A}$-partite systems, say of $\str{B}_1$ into $\str{B}_2$, are defined as the isomorphisms and embeddings of structures together with the condition that all parts 
are being preserved (the part $X^i_{\str{B}_1}$ is mapped to $X^i_{\str{B}_2}$  for every $i = 1,2,\ldots,a$).  Of course, $\str{A}$ itself can be considered as an $\str{A}$-partite system.

We say that an $\str{A}$-partite $L$-system is {\em transversal} if all of its parts consist of at most one vertex. 
Thus $\str{A}$-partite $L$-systems are naturally ordered by ordering of its parts.

\begin{lem}[Partite Lemma with relations and functions]\label{partlem}
Let $L$ be a language, $\str{A}$ be a finite $L$-structure, and $\str{B}$ be a finite $\str{A}$-partite $L$-system.
Then there exists a finite $\str{A}$-partite $L$-system $\str{C}$ such that 
$$
\str{C}\longrightarrow (\str{B})^\str{A}_2.
$$
Moreover if every irreducible subsystem of $\str{B}$ is transversal, then we can also ensure that every irreducible subsystem of $\str{C}$ is transversal.
\end{lem}
(Compare the Partite Lemma in~\cite{Hubicka2016}. Here we introduce the statement about transversality of irreducible substructures.
Note that the embeddings considered in the Ramsey statement are all as $\str{A}$-partite systems.)

Advancing the proof of Lemma~\ref{partlem},
for completeness, we briefly recall the Hales-Jewett Theorem~\cite{Hales1963}.
Consider the family of functions $f:\{1,2,\ldots, N\} \to \Sigma$ for some finite alphabet $\Sigma$. A \emph{combinatorial line} $\mathcal L$ is a pair $(\omega,h)$ where $\emptyset\neq\omega\subseteq \{1,2,\ldots, N\}$ and $h$ is a function from $\{1,2,\ldots, N\}\setminus \omega$ to $\Sigma$. The combinatorial line $\mathcal L$ describes the family of all those functions $f:\{1,2,\ldots, N\} \to \Sigma$ that are constant on $\omega$ and $f(i)=h(i)$ otherwise. The Hales-Jewett Theorem guarantees, for sufficiently large $N$, that for every $2$-colouring of the functions $f:\{1,2,\ldots, N\} \to \Sigma$ there exists a monochromatic combinatorial line.

\begin{proof}[Proof of Lemma~\ref{partlem}]
Assume without loss of generality  $A = \{1, 2,\ldots, a\}$  and denote by $\mathcal X_\str{B} = \{X^1_\str{B},X^2_\str{B},\ldots, X^a_\str{B}\}$ the parts of $\str{B}$.
We take $N$ sufficiently large (that will be specified later) and construct an $\str{A}$-partite $L$-system  $\str{C}$
with parts $\mathcal X_\str{C} = \{X^1_\str{C},X^2_\str{C},\ldots, X^a_\str{C}\}$  as follows:
\begin{enumerate}
\item For every $1\leq i\leq a$ let $X_\str{C}^i$ be the set of all functions $$f:\{1,2,\ldots,N\}\to X_\str{B}^i.$$
\item For every relational symbol $\rel{}{}\in L_\mathcal R$, put $$(f_1,f_2,\ldots, f_{\arity{}})\in \rel{C}{}$$ if and only if for every $1\leq j\leq N$ it holds that $$(f_1(j),f_2(j),\ldots, f_{\arity{}}(j))\in \rel{B}{}.$$
\item
 For every function symbol $\func{}{}\in L_\mathcal F$ we put $$\func{C}{}(f_1,f_2\ldots, f_{\arityf{})})\neq \emptyset$$
if and only if there exists $r\geq 1$ and a set $\{i_1,i_2,\ldots, i_r\}\subseteq A$ such that for every $1\leq j\leq N$ it holds that $\func{B}{}(f_1(j),f_2(j),\ldots, f_{\arityf{}}(j))$ is defined and its value contains vertices precisely in parts $i_1,i_2,\ldots, i_{r}$. In this case the value of $\func{C}{}(f_1,f_2\ldots, f_{\arityf{}})$ consists of one vertex in each of parts $i_1,i_2,\ldots, i_{r(F)}$.
The vertex in part $i_k$ is the function $f_k:\{1,2,\ldots,N\}\to X_\str{B}^{i_k}$ defined by putting $f_k(j')$ to be the (unique) vertex in $\func{B}{}(f_1(j'),f_2(j'),\allowbreak \ldots, f_{\arityf{}}(j'))\cup X_\str{B}^{i_k}$ (for each $1\leq j'\leq N$).
\end{enumerate}
This completes the construction of $\str{C}$.
It is easy to check that $\str{C}$ is  indeed an $\str{A}$-partite $L$-system with parts $\mathcal X_\str{C} = \{X^1_\str{C},\allowbreak X^2_\str{C},\allowbreak \ldots,\allowbreak  X^a_\str{C}\}$.

\medskip

We verify that, if $N$ is large enough,  $\str{C}\longrightarrow (\str{B})^\str{A}_2$.
Let $\widetilde{\str{A}}_1, \widetilde{\str{A}}_2,\ldots, \widetilde{\str{A}}_t$ be an enumeration of all subsystems of $\str{B}$ which are isomorphic to $\str{A}$. 
Put  $\Sigma=\{1,2,\ldots, t\}$ which we consider as an alphabet.
Each combinatorial line $\mathcal L=(\omega, h)$ in $\Sigma^N$ corresponds to an embedding $e_\mathcal L:\str{B}\to \str{C}$ which assigns to every vertex $v\in X^p_\str{B}$ a function $e_\mathcal L(v):\{1,2,\ldots, N\}\to X^p_\str{B}$ (i.e. a vertex of $X^p_\str{C}$) such that:
$$e_\mathcal L(v)(j)=
 \begin{cases} 
    \hbox{$v$ for $j\in \omega$, and,}\\
    \hbox{the unique vertex in $\widetilde A_{h(j)}\cap X^p_\str{B}$ otherwise.}
   \end{cases}
$$
It follows from the construction of $\str{C}$ and from the fact that $\str{B}$ has a projection to $\str{A}$ that  $e_\mathcal L$ is an embedding.

 Let $N$ be the Hales-Jewett number guaranteeing a monochromatic line in any $2$-colouring of the $N$-dimensional cube over an alphabet $\Sigma$.
Now assume that $\mathcal{A}_1, \mathcal{A}_2$ is a $2$-colouring of all copies of $\str{A}$ in $\str{C}$.
Using the construction of $\str{C}$ we see that  among copies of $\str{A}$ are copies induced by an $N$-tuple    $(\widetilde{\str{A}}_{u(1)}, \widetilde{\str{A}}_{u(2)},\ldots,\widetilde{\str{A}}_{u(N)})$ of copies of $\str{A}$ in $\str{B}$ for every function $u:\{1,2,\ldots,\allowbreak N\}\to \{1,2,\ldots, t\}$. However such copies are coded by the elements of the cube $\{1,2,\allowbreak \ldots,\allowbreak t\}^N$ and thus there is a monochromatic 
combinatorial line $\mathcal L$. The monochromatic copy of $\str{B}$ is then $e_\mathcal L(\str B)$.

\medskip

Finally we verify that if every irreducible subsystem of $\str{B}$ is transversal
then also every irreducible subsystem of $\str{C}$ is transversal. Assume the contrary
and denote by $\str{D}$ a non-transversal irreducible subsystem of $\str{C}$.
Denote by $f_1,f_2,\ldots, f_n$ an enumeration of all distinct vertices of $\str{D}$. By non-transversality
assume that $f_1$ and $f_2$ are in the same part.  For every $j\in \{1,2,\ldots, N\}$
denote by $\str{D}_j$ the  substructure of $\str{B}$ on vertices $f_1(j),f_2(j),\ldots, f_n(j)$.
Because $\str{D}_j$ is a homomorphic image of $\str{D}$ it is irreducible and thus it follows that $\str{D}_j$ is transversal.
Consequently $f_1(j)=f_2(j)$.  Because this holds for every choice of $j$, we have $f_1=f_2$. A contradiction.
\end{proof}
\begin{proof}[Proof of Theorem~\ref{thm:main}]
Given $\str{A},\str{B}\in \vv{\K}\subseteq \vv{\Str}(\vv{L})$ use Theorem~\ref{thm:models} to obtain
$$\str{C}_0\longrightarrow (\str{B})^\str{A}_2.$$

Enumerate all copies of $\str{A}$ in $\str{C}_0$  as  $\{\widetilde{\str{A}}_1, \widetilde{\str{A}}_2,\ldots, \widetilde{\str{A}}_b\}$.
We will define $\str{C}_0$-partite systems (`pictures') $\str{P}_0, \str{P}_1, \ldots, \str{P}_b$ such that
\begin{enumerate}
\item[(i)] every irreducible subsystem $\str{E}$ of $\str{P}_k$ is transversal and (if seen as a structure) is isomorphic to some substructure of $\str{B}$, and,
\item[(ii)] in any $2$-colouring of ${\str{P}_k}\choose{\str{A}}$, $1\leq k\leq b$, there exists a copy $\widetilde{\str{P}}_{k-1}$ of $\str{P}_{k-1}$ such that all copies of $\str{A}$ with a projection to $\widetilde{\str{A}}_k$ are monochromatic.
\end{enumerate}
We then show that putting $\str{C}$ to be $\str{P}_b$ (seen as a structure) with the linear order completed arbitrarily (extending the order of the parts), we have the desired Ramsey property $\str{C}\longrightarrow (\str{B})^{\str{A}}_2$.

We first verify that if (i) holds, then $\str{C}\in \vv{\K}$. Assume the contrary. Denote by $\str{F}$ the minimal substructure of $\str{C}$ such that $\str{F}\notin \vv{\K}$. Because $\K$ is
a free amalgamation class, we know that the unordered reduct of $\str{F}$ is irreducible. By (i) however $\str{F}$ is a substructure of $\str{B}\in \K$. A contradiction to  $\K$ being hereditary.

It remains to prove (i) and (ii).
Put $C_0 = \{1,\ldots,c\}$ and  $\mathcal X_{\str{P}_k} = \{X_k^1, X_k^2,\allowbreak \dots,\allowbreak X_k^c\}$.
We proceed by induction on $k$.

\begin{enumerate}
\item
\begin{figure}[t]
\centering
\includegraphics{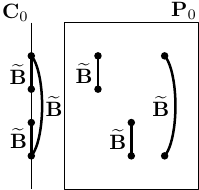}
\caption{The construction of $\str{P}_0$.}
\label{fig:picture0}
\end{figure}
The  picture $\str{P}_0$ is constructed as a disjoint union of  copies of $\str{B}$. For every copy $\widetilde{\str{B}}$ of
$\str{B}$ in $\str{C}_0$ we include a  copy $\widetilde{\str{B}}'$ of $\widetilde{\str{B}}$ in $\str{P}_0$ which projects onto $\widetilde{\str{B}}$. The copies corresponding to different $\widetilde{\str{B}}$ are disjoint (see Figure~\ref{fig:picture0}).
This indeed satisfies (i).

\item
Suppose the picture $\str{P}_k$ is already constructed.  
 Let $\str{B}_k$ be the substructure of  $\str{P}_k$ induced by $\str{P}_k$ on vertices which project to $\widetilde{\str{A}}_{k+1}$.
We use the Partite Lemma~\ref{partlem} to obtain
an $\widetilde{\str{A}}_{k+1}$-partite system $\str{D}_{k+1}$ with $\str{D}_{k+1}\longrightarrow (\str{B}_k)^{\widetilde{\str{A}}_{k+1}}_2$.
Now consider all copies of 
$\str{B}_k$ 
 in $\str{D}_{k+1}$ and extend each of these structures to a copy of $\str{P}_k$ (using free amalgamation of $\str{C}_0$-partite systems). These copies are disjoint outside $\str{D}_{k+1}$.  In this extension we preserve the parts of all the copies.
The result of this multiple amalgamation is $\str{P}_{k+1}$. Because $\str{D}_{k+1}\longrightarrow (\str{B}_k)^{\widetilde{\str{A}}_{k+1}}_2$ we know that $\str{P}_{k+1}$ satisfies (ii).

From the `Moreover' part of Lemma~\ref{partlem}, we may assume that $\str{D}_{k+1}$ satisfies (i).
Because $\str{P}_{k+1}$ is created by a series of free amalgamations, it follows that $\str{P}_{k+1}$ also satisfies (i).
\end{enumerate}
Put $\str{C} = \str{P}_b$.  It  follows easily that  $\str{C}\longrightarrow
(\str{B})^{\str{A}}_2$. Indeed, by a backward induction on $k$ one proves  that in any
$2$-colouring of ${\str{C}}\choose{\str{A}}$ there exists a copy
$\widetilde{\str{P}}_0$ of $\str{P}_0$ such that the colour of a copy of
$\str{A}$ in $\widetilde{\str{P}}_0$ depends only on its projection to $\str{C}_0$.  As this in turn induces a 
colouring of the copies of $\str{A}$ in $\str{C}_0$, we obtain a monochromatic copy
of $\str{B}$ in $\widetilde{\str{P}}_0$.
\end{proof}

\section{Ordering property}
\label{sec:ordering}
For many Ramsey classes $\K$, the class $\vv{\K}$ of free orderings of structures in $\K$ has the ordering property.
In fact, for free amalgamation classes we can give a full characterisation by means of the
following easy (and folklore) proposition.
\begin{prop}
\label{prop:freeorder}
Let $\K$ be a free amalgamation class of $L$-structures. Then $\vv{\K}$ has the 
ordering property if and only if all closures $\cl_\str{A}(u)$ of vertices are mutually isomorphic single element structures.
\end{prop}
\begin{remark}
It does not follow from Proposition~\ref{prop:freeorder} that the class $\K$ must
have only trivial closures in order for $\vv{\K}$ to have the ordering
property.  Consider, for example, a class of structures with one binary function
$\func{}{}$ such that the images of tuples with
duplicated vertices are empty. Here all vertices are closed, but pairs of vertices have
non-trivial closures. A related example is discussed in Section~\ref{sec:steiner}.
\end{remark}
\begin{proof}
Assume that there is $\ostr{A}\in \K$ and $u,v\in \ostr{A}$ such that
$\cl_\str{A}(u)$ is not isomorphic to $\cl_\str{A}(v)$. Then one can choose an 
ordering $\ostr{A}$ of $\str{A}$ so  that the set of all vertices $v'\in
A$ such that $\cl_\str{A}(v')$ is isomorphic to $\cl_\str{A}(v)$ forms an
initial segment.  Now assume, to the contrary, that there exists $\str{B}\in
\K$ such that every ordering of $\str{B}$ contains a copy of $\ostr{A}$. This
is clearly not possible because one can choose an ordering of $\str{B}$ such that
all vertices  $u'\in B$ such that  $\cl_\str{A}(u')$ is isomorphic to
$\cl_\str{A}(u)$ forms an initial segment. This is a contradiction to $\vv{\K}$ having the  ordering property.

Now consider the case that all closures of vertices are isomorphic, but not trivial.
Because the intersection of two closures is also closed, it follows that closures of vertices are disjoint and thus it is always possible to order structures in a way that all vertex closures form
intervals.  It follows that there is no $\str{B}$ witnessing the 
ordering property for any $\ostr{A}\in \K$ which is ordered so that some
vertex closure is not an interval.

Finally assume that $\K$ is a free amalgamation class where all closures of vertices
are trivial and mutually isomorphic.  We may assume that there are no unary relations. Given $\ostr{A}\in \vv{K}$ we
construct $\ostr{B}_0$ from a disjoint copy of $\ostr{A}$ and $\vvrev{\str{A}}$ (by this we mean
a structure created from $\ostr{A}$ by reversing the linear ordering of vertices).
Now extend $\ostr{B}_0$ to $\ostr{B}_1$ by adding,  between every neighbouring pair
of vertices $u\leq_{\ostr{B}_0} v\in B_0$ a new vertex $n_{u,v}$. Extend the order so $u\leq_{\ostr{B}_1} n_{u,v}\leq_{\ostr{B}_1} v$. By free amalgamation, $\ostr{B}_1 \in \vv{\K}$.

Denote by $\ostr{I}\in \vv{\K}$ a structure consisting of two vertices $u\leq_\ostr{I} v$ and no relations or functions containing both of them besides $\leq_\ostr{I}$
(such a structure can be created by means of the free amalgamation of two copies of the unique single vertex structure in $\str{K}$ over an empty set with order completed arbitrarily and thus clearly $\ostr{I}\in \vv{\K}$).
 Find $\ostr{B}\in \vv{\K}$ with $\ostr{B}\longrightarrow(\ostr{B}_1)^{\ostr{I}}_2$ by application of Theorem~\ref{thm:main}.

We verify that $\ostr{B}$ has the desired property.
Let $\ostr{C}$ be any re-ordering of $\ostr{B}$. This re-ordering induces a  coloring of $\ostr{B}\choose \ostr{I}$: if the order of the points in some $\ostr{I}' \in {\ostr{B}\choose \ostr{I}}$ agrees with the order in $\ostr{C}$, then  color
$\ostr{I}'$ red,  and blue otherwise.  The monochromatic copy of $\ostr{B}_1$ will have the property that it is either ordered in the same way as $\ostr{B}_1$ or the order is reversed.
By construction of $\ostr{B}_0$ it follows that in both alternatives there is a copy of $\ostr{A}$.
\end{proof}

In the following we generalize the main idea of this proof (the idea of which goes back to~\cite{Nesetril1975}) to classes with
non-trivial closures of vertices.  Free orderings do not suffice anymore and we have to define carefully the admissible
orderings. 

\subsection{Admissible orderings}

\begin{defn}
\label{def:closure-components}
Let $\str{A}$ be an $L$-structure.  If $a, b \in A$ we write $a \sim_A b$ if $\cl_\str{A}(a) = \cl_\str{A}(b)$. This is an equivalence relation on $A$ and we refer to the classes as the \emph{closure-components} of $\str{A}$. The class containing $a$ will be denoted by $\Scc_\str{A}(a)$. 

If $a \in A$ we define the \emph{level} $l_\str{A}(a)$ of $a$ in $\str{A}$ inductively. We say that $l_\str{A}(a) = 0$ in the case where $\Scc_\str{A}(a) = \cl_\str{A}(a)$; otherwise $l_\str{A}(a) = l_\str{A}(b) +1$ where $b$ is a vertex of the maximal level in $\str{A}$ amongst vertices in $\cl_\str{A}(a)\setminus \Scc_\str{A}(a)$. 

We say that $\str{A} \in \K$ is a \emph{closure-extension}
at level $k$ if there is a unique closure-component $C$ of vertices of level $k$ in $\str{A}$ and $\cl_\str{A}(a) = \str{A}$ for every $a \in C$. In this case, we write $A^\circ = A \setminus C$.  Every closure of a vertex is a closure-extension.

We say that two closure-components $C$ and $C'$ of $\ostr{A}$ (or their closures) are {\em homologous} if $\cl_\str{A}(C)$ and $\cl_\str{A}(C')$ are isomorphic and this isomorphism is an identity on $\cl_\str{A}(C)\setminus C=\cl_\str{A}(C')\setminus C'$. Note that if $C \neq C'$, then $\cl_\str{A}(C)\setminus C = \cl_\str{A}(C)\cap \cl_\str{A}(C')$ is closed in $\str{A}$.
\end{defn}
\begin{example}
Consider the class $\mathcal T$ of forests represented by a single unary function (the predecessor relation) described in Section~\ref{sec:orderingpropintro}.
In this class the closure of a vertex is the path to a root vertex. Every vertex thus forms a trivial closure component
and its level is determined by the distance to the root vertex.
In this particular case a structure is a closure-extension if and only if
$\str{A}$ is an oriented path (that is structure on vertex set $\str{A}=\{v_1,v_2,\ldots, v_n\}$ and
$\func{A}{}(v_i)=v_{i+1}$ for every $1\leq i<n$).

\medskip

\begin{figure}
\centering
\includegraphics{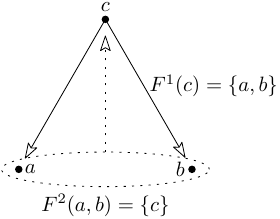}
\caption{Example of a closure-extension $\str{A}$ where $A^\circ$ is not a substructure.}
\label{notclosedfig}
\end{figure}
Observe also that for a closure-extension $\str{A}$ (depicted in Figure~\ref{notclosedfig}), the set of vertices $A^\circ$ is not necessarily closed.
Consider a language with a function $\func{}{1}$ arity 1 a 
function $\func{}{2}$ of arity 2. The structure $\str{A}$ with $A=\{a,b,c\}$,
$\func{A}{1}(c)=\{a,b\}$, $\func{A}{2}(a,b)=\{c\}$ is a closure-extension of level 1: $l(a)=l(b)=0$ and $l(c)=1$, the set $A^\circ$ is $\{a,b\}$, however $\cl_\str{A}\{a,b\}=\{a,b,c\}$.
\end{example}

Suppose $\str{A}, \str{B} \in \K$ are  closure-extensions and $\vv{\str{A}}, \vv{\str{B}} \in \vv{\K}$ are  orderings. We say that these are \emph{similar} if there is an isomorphism $\alpha:\str{A} \to \str{B}$ which is also order-preserving when seen as a mapping $\alpha:A^{\circ}\to B^{\circ}$. This is an equivalence relation and in our admissible orderings, we choose a fixed representative from each similarity-type of ordered closure-extension.

In what follows we shall assume that we have fixed some total ordering $\ostr{A} \trianglelefteq \ostr{B}$ between isomorphism types of orderings of closure-extensions such that
\begin{enumerate}[label=S\arabic*]
 \item $\lvert A\rvert <\lvert B\rvert$ implies $\ostr{A} \trianglelefteq \ostr{B}$.
\end{enumerate}
 (In particular, $\trianglelefteq$ is a well ordering.)

First we define a preorder of vertices which we will later refine to a linear order.
Given two vertices $u\neq v\in \ostr{A}$ we write $u\preccurlyeq_\ostr{A} v$ if one of the following holds:
\begin{enumerate}[label=P\arabic*]
 \item $\cl_{\ostr{A}}(u)\trianglelefteq \cl_{\ostr{A}}(v)$ and they are not isomorphic;
 \item $\cl_{\ostr{A}}(u)$ is isomorphic to $\cl_{\ostr{A}}(v)$ but $\cl_{\ostr{A}}(u)\setminus \Scc_{\ostr{A}}(v)$ is lexicographically before $\cl_{\ostr{A}}(v)\setminus \Scc_{\ostr{A}}(v)$ considering the order $\leq_\ostr{A}$;
 \item $\cl_{\ostr{A}}(u)$ and $\cl_{\ostr{A}}(v)$ are homologous closure-extensions. 
\end{enumerate}

Note that this is indeed a preorder on the vertices of $\str{A}$. We can now describe our class of admissible orderings.

\begin{defn}\label{def:admissible} Suppose $\K$ is a free amalgamation class. We say that $\mathcal O \subseteq \vv{\K}$ is a \emph{class of admissible orderings} of structures in $\K$ if  the following conditions hold.
\begin{enumerate}[label=A\arabic*]
\item \label{o:extend} If $\str{A} \in \K$, then there is some ordering $\leq_\ostr{A}$ of $\str{A}$ in $\mathcal O$.
\item \label{o:substructures} $\mathcal O$ is closed for substructures.
\item \label{o:rootorder}
For every $\ostr{A}\in \mathcal O$, the ordering $\leq_\ostr{A}$ refines $\preccurlyeq_\ostr{A}$.
\item \label{o:interval}
For every $\ostr{A}\in \mathcal O$, the closure-components form linear intervals in $\leq_\ostr{A}$.
\item \label{o:extend2}
For every $\str{B}\in \K$, if  $\str{A}_1,\str{A}_2,\ldots, \allowbreak \str{A}_n$ is a family of substructures and  $\leq$ is a linear order of
$A=\cup_{1\leq i\leq n} A_i$ such that
\begin{enumerate}
\item $\leq $ satisfies the conclusions of \ref{o:rootorder} and~\ref{o:interval}; and
\item each substructure of $\str{B}$ contained in $A$ is admissibly ordered by $\leq$;
\end{enumerate}
then there exists $\ostr{B}\in \mathcal O$ such that $\leq_{\ostr{B}}$ restricted to $A$ is $\leq$.
\item \label{o:unique} 
Suppose that $\ostr{A}, \ostr{B} \in \mathcal O$ are similar ordered closure-extensions. Then $\ostr{A}$ is isomorphic to $\ostr{B}$.
\end{enumerate}
\end{defn}

\begin{example}
Consider $\str{A}\in \mathcal T$ where $\mathcal T$ is the class of forests discussed
in Section~\ref{sec:orderingpropintro}. Because the closure-extensions are all formed by oriented
paths, the order $\trianglelefteq$ requires vertices to be ordered according to their levels (in particular all root vertices come first and can be ordered arbitrarily).
The sons of a vertex $v\in \str{A}$  are trivial homologous components and thus they are required
to always form an interval and the order amongst these intervals is given by order of their
fathers.

There is some flexibility in the definition of admisibility. For example
it is possible to order forests in a way that every tree (and recursively every subtree) forms an interval.
The particular choice is however not very important as it can be shown that they are all equivalent up
to bi-definability (this follows as $\mathcal O$ is a Ramsey class, see \cite{Kechris2005}).
\end{example}
\begin{prop} \label{prop:ordexists} Suppose $\K$ is a free amalgamation class. Then there is a class $\mathcal O$ of admissible orderings of $\K$. 
\end{prop}

\begin{proof}
We proceed by induction on $\vv{\K}$ ordered arbitrarily in order of increasing number of vertices. In this order, for every $\ostr{C}\in \vv{\K}$ we decide
if $\ostr{C}\in \mathcal O$ or not by a variant of a greedy algorithm. In the induction step assume that we already decided the presence in  $\mathcal O$ for every proper substructure
of $\ostr{C}$. 

We put $\ostr{C}\in \mathcal O$ if and only if:
\begin{enumerate}[label=O\arabic*]
 \item \label{cond1} $\ostr{C}$ satisfies \ref{o:rootorder} and \ref{o:interval},
 \item \label{cond2} every proper substructure of $\ostr{C}$ is in $\mathcal O$, and,
 \item \label{cond3} if $\ostr{C}$ is an ordered closure-extension, then there is no similar but non-isomorphic $\ostr{D}\in \mathcal O$.
\end{enumerate}
This finishes the description of $\mathcal O$.  We verify that the conditions of Definition~\ref{def:admissible} are satisfied.

\medskip

First we check \ref{o:extend2}. 
Assume, to the contrary, that there are  $\str{B}\in \K$, substructures $\str{A}_1$, $\str{A}_2, \ldots, \str{A}_n$ and a linear order $\leq$ on $A=\cup_{1\leq i\leq n} A_i$ without such an extension. From all counter-examples choose one minimizing $\lvert B\rvert$ and among those, minimize $\lvert B\lvert - \lvert A\lvert$.

We consider three cases:
\begin{enumerate}
\item Suppose $\str{B}$ is not a closure-extension and $B=A$. It follows that the order $\leq$ satisfies both~\ref{cond1} and~\ref{cond2} and thus we can put $\leq_{\ostr{B}}=\leq$ and obtain $\ostr{B}\in \mathcal O$, a contradiction.
\item Suppose $\str{B}$ is a closure-extension and $B^\circ=A$.
Extend $\leq$ to an order of $\ostr{B}$ in a way it satisfies \ref{cond1}.
In this case $\ostr{B}$ satisfies \ref{cond2} because no proper substructure contains $B\setminus B^\circ$. Furthermore, we may assume that \ref{cond3} holds. This is a contradiction to $\ostr{B}\notin \mathcal O$.
\item Suppose neither of the previous cases apply. Then there is a proper closure-extension $\str{C} \subseteq\str{B}$ such that $C\not \subseteq A$.  Denote by $\leq'$ the 
order $\leq$ restricted to $C \cap A$ and consider the structures $\str{C}_i=\str{A}_i\cap \str{C}$, $1\leq i\leq n$.  From the minimality of the counter-example
and because $\lvert C\rvert <\vert B\rvert$ we know that order of $\leq'$ can be extended to an admissible order $\ostr{C}$ of $\str{C}$.
Now extend to $C \cup A$ the orders $\leq$ and $\leq_\ostr{C}$ in such a way that \ref{o:rootorder} and~\ref{o:interval} are satisfied.  This combined
order along with the family of substructures $\str{C}$, $\str{A}_1$, $\str{A}_2, \ldots, \str{A}_n$ contradicts the second assumption about the minimality of
the counter-example.  
\end{enumerate}
This finishes proof of~\ref{o:extend2}.

\medskip
Condition \ref{o:extend} is implied by \ref{o:extend2}.
Conditions \ref{o:substructures}, \ref{o:rootorder}, \ref{o:interval} and \ref{o:unique} follows directly from the construction of $\mathcal O$.
\end{proof}

\subsection{Proof of Theorem~\ref{thm:ordering}}

Advancing the proof of Theorem~\ref{thm:ordering} we show two lemmas which generalize
the main ideas of proof of Proposition~\ref{prop:freeorder}.

\begin{lem}
\label{lem:type}
Let $\mathcal R$ be a Ramsey class of ordered structures.
Then for every $\ostr{A}\in \mathcal R$ there exists $\ostr{B}\in \mathcal R$ such that
every re-ordering $\ostr{C}\in \mathcal R$ of $\ostr{B}$ contains a re-ordering $\widetilde{\str{A}}$ of $\ostr{A}$ such that every two
isomorphic substructures of $\ostr{A}$ are also isomorphic in the order of $\widetilde{\str{A}}$.
\end{lem}
\begin{proof}
Let $\ostr{A}$ be a structure and $\ostr{A}_0$ a substructure of $\ostr{A}$. Denote by $n$ the number of possible re-orderings of $\ostr{A}_0$ (i.e. $n=\lvert A_0\lvert !$).
 By the Ramsey property construct $\ostr{B}_0$ such that $\ostr{B}_0\longrightarrow (\ostr{A})^{\ostr{A}_0}_n$. Every re-ordering of $\ostr{B}_0$ induces an $n$-coloring of copies of $\ostr{A}_0$ in $\ostr{B}_0$
and so by the Ramsey property there exists a copy $\widetilde{\str{A}}$ of re-ordered $\ostr{A}$ in $\ostr{B}_0$ having
the property that all copies of $\ostr{A}_0$ in $\widetilde{\str{A}}$ are
ordered the same way. The statement follows by iterating this argument for every substructure
of $\ostr{A}$.
\end{proof}
\begin{lem}
\label{lem:homologous}
Let $\mathcal K$ by a free amalgamation class.  Then for every $\ostr{A}\in \vv{\K}$
in which  every closure-component forms an interval (in the order of $\ostr{A}$) there exists $\ostr{B}$ such that 
every re-ordering $\ostr{C}$ of $\ostr{B}$ where every closure-component forms an interval (in the order of $\ostr{C}$) contains a copy of a re-ordering of $\ostr{A}$
where the order between vertices in distinct homologous closure-components is preserved.
\end{lem}
\begin{proof}
For simplicity we show the construction for a given $\ostr{A}$ and two
distinct homologous closure-components $C_1\leq_{\ostr{A}} C_2$.  The full statement
can be shown by iterating the argument for every such pair. The construction is schematically depicted in Figure~\ref{constructionfig}.
\begin{figure}
\centering
\includegraphics{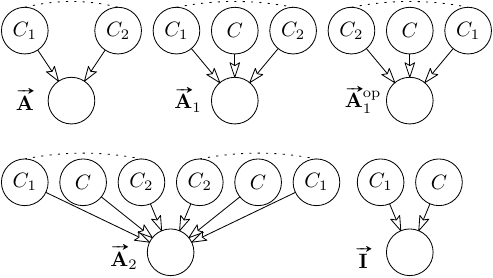}
\caption{Construction used in the proof of Lemma~\ref{lem:homologous}. Pairs of closure components which are not necessarily free are connected by dotted lines.}
\label{constructionfig}
\end{figure}

First observe that $C_1$ and $C_2$ are not in the closure of $\cl_\str{A}(C_1)\setminus C_1 = \cl_\str{A}(C_2)\setminus C_2$
because this would imply  $C_1=C_2$.

Now, by a free amalgamation construct $\ostr{A}_1$ extending $\ostr{A}$ with a new closure-component $C$ homologous to both $C_1$ and $C_2$
where the  linear order is extended so that $C_1\leq_\ostr{A} C\leq_\ostr{A} C_2$. 
Denote by $\vvrev{\str{A}}_1$ the structure created from $\ostr{A}_1$ by re-ordering closure-components $C_1$, $C$ and $C_2$ so that
$C_2\leq_{\vvrev{\str{A}}_1} C\leq_{\vvrev{\str{A}}_1} C_1$.
Denote by $\ostr{A}_2$ the free amalgamation of $\ostr{A}_1$ and $\vvrev{\str{A}}_1$ over $\cl_\str{A}(C_1)\setminus C_1$. Because $\K$ is a free amalgamation class we have
$\ostr{A}_1, \vvrev{\str{A}}_1, \ostr{A}_2\in \vv{\K}$.

Note that the addition of $C$ into $\str{A}_1$ is necessary only when $\cl(\ostr{A}(C_1\cup C_2))$ is not a result of free amalgamation of two copies of $\cl_{\ostr{A}(C_1)}$.
In general there may be relations spanning $C_1$ and $C_2$ which would make use of the Ramsey argument below impossible.

Put $\ostr{I}=\cl_{\ostr{A}_1} (C_1\cup C)$ and use Theorem~\ref{thm:main}  to find $\ostr{B} \in \vv{\K}$ containing $\ostr{A}_2$ such that $\ostr{B}\longrightarrow (\ostr{A}_2)^\ostr{I}_2$. 
Now every re-ordering $\ostr{C}$ of $\ostr{B}$ induces a $2$-coloring of the  copies of $\ostr{I}$ in $\ostr{B}$ which in turn leads to
the existence of a copy of a re-ordering of $\ostr{A}$ where the order of $C_1$ and $C_2$ is preserved.
\end{proof}
Now we are finally ready to prove the main result of this section.

\begin{proof}[Proof of Theorem~\ref{thm:ordering}]
Given $\ostr{A}\in \mathcal O$ we construct $\ostr{B}\in \vv{\K}$ with $\ostr{A}$ as a substructure 
such that every admissible re-ordering $\ostr{C}\in \mathcal O$ of $\ostr{B}$
contains a copy of $\ostr{A}$.

By application of Lemmas~\ref{lem:homologous} and \ref{lem:type} it is enough to construct  an admissibly ordered $\ostr{B}_0$ such that  that every re-ordering $\ostr{C}_0\in \mathcal O$ of $\ostr{B}_0$ with the following two properties contains a copy of $\ostr{A}$:
\begin{enumerate}
\item[(a)] the order of vertices in distinct  homologous components is preserved; 
 \item[(b)]  the order in $\ostr{C}_0$ on substructures which are closures of vertices depends only on their isomorphism type in $\ostr{B}_0$.
\end{enumerate}

Given $\ostr{A}$ denote by $\ostr{A}_1$, $\ostr{A}_2$,\ldots, $\ostr{A}_n$ all admissible re-orderings of $\ostr{A}$ 
having properties (a) and (b).  So these structures all have the same domain, and only differ in their orderings.
 Put $\ostr{B}_0$ to be the disjoint union of  $\ostr{A}_1$, $\ostr{A}_2$,\ldots, $\ostr{A}_n$ with the order completed arbitrarily. Let  $\ostr{C}_0$ be an admissible reordering of $\ostr{B}_0$ satisfying (a) and (b). Denote by $\alpha_i : \ostr{A} \to \ostr{C}_0$ the map which sends $a \in A$ to $a \in A_i$, for $i \leq n$. It is enough to prove the following.

\medskip

\noindent\textit{Claim:\/} For all $k\geq 0$,  there is some $i \leq n$ such that  the map $\alpha_i$ preserves the ordering on vertices in $\ostr{A}$ of level at most $k$.

\medskip

We prove this by induction on $k$. Denote by $A\vert_k$ the set of vertices in $A$ at level $\leq k$ (this is not necessarily a substructure or even a structure in $\mathcal{K}$, but it still makes sense to discuss admissible orderings of it as it contains the closures of all of its vertices, and the criteria for being an admissible ordering depend only on these). By setting $A\vert_{-1} = \emptyset$, we can incorporate the proof for the base case $k = 0$ into the general argument. 

\medskip

\noindent\textit{Step 1:\/} Every admissible ordering of vertices in $A\vert_{k-1}$ satisfying (a) and (b) extends to one of $A\vert_{k}$. If $\str{X}$ is the closure of a vertex of level $k$ in $\str{A}$, then any two such closure-extensions differ on $\str{X}$ by a permutation in $\Aut(\str{X}/ X^\circ)$ (that is, automorphisms of $\str{X}$ fixing all vertices of level less than $k$). 

Indeed, we need only to say how to define the ordering on $\str{X}$. But, given the ordering on $X^\circ$, this is determined by condition~\ref{o:unique} in Definition~\ref{def:admissible}, up to the action of $\Aut(\str{X}/ X^\circ)$. It is easy to see that the resulting ordering on $A\vert_{k}$ is admissible.

\medskip

\noindent\textit{Step 2:\/} Suppose the claim holds up to level $k-1$. Let $I$ denote the set of $i \leq n$ for which $\alpha_i$ restricted to $A\vert_{k-1}$ is order-preserving. So this is non-empty. Let $i \in I$ and let $\str{X}$ be as in Step 1. 

There is $\beta_i \in \Aut(\str{X}/ X^\circ)$ such that $\alpha_i\circ \beta_i$ preserves the ordering on $\str{X}$ and so is an isomorphism between $\ostr{X}$, the structure on $\str{X}$ in $\ostr{A}$ and the substructure $\alpha_i(\str{X})$ in $\ostr{C}_0$. By Step 1, there is some $j \leq n$ such that the map $\beta_i^{-1}$, regarded as a map from $\ostr{A}_i$ to $\ostr{A}_i$ (in $\ostr{B}_0$) is order-preserving and all vertices in $A\vert_{k-1}$ have the same ordering in $\ostr{A}_i$ and $\ostr{A}_j$.  By condition (b), it follows that this map between the corresponding subsets of $\ostr{C}_0$ is order-preserving (as all orderings are determined by what happens in closures of vertices). Thus $j \in I$ and $\alpha_j$ is order-preserving on $X$. Repeating this argument for other vertices at level $k$, we complete the inductive step.

We have verified that the class of admissible orderings $\mathcal O$ has the ordering property. The Ramsey property follows from Theorem~\ref{thm:main}
and \ref{o:extend2}: Given $\ostr{A},\ostr{B}\in \mathcal O$ and $\ostr{C}\in \vv{\K}$ such that $\ostr{C}\longrightarrow(\ostr{B})^\ostr{A}_2$.
Construct a new order $\leq$ of $\ostr{C}$ in a way that $\leq$ agrees with $\leq_\str{C}$ on every copy of $\ostr{B}$ in $\ostr{C}$. Complete $\leq$ so it satisfies
\ref{o:rootorder} and \ref{o:interval}. By \ref{o:extend2} it follows that $\str{C}$ ordered by $\leq$ is in $\mathcal O$.
\end{proof}

\section{Irreducible substructure faithful EPPA}
\label{sec:EPPA}

The proof of Theorem~\ref{EPPA} is a variant of the proof of clique-faithful EPPA by Hodkinson and Otto~\cite{hodkinson2003}
strengthened for languages with unary functions.  As in~\cite{hodkinson2003} our starting point is the following construction giving EPPA for relational
structures and we verify coherency as in~\cite{Siniora2, Siniora}.

\begin{thm}[Herwig \cite{Herwig1995}, coherency verified by Solecki~\cite{solecki2009}]
\label{thm:Herwig}
For any relational language $L$, the class of all finite $L$-structures has coherent EPPA.
\end{thm}

To apply this construction to structures with functions we will temporarily interpret
functions as relational symbols.

\begin{defn} Suppose $L$ is a language where all function symbols are unary. 
Given an $L$-structure $\str{A}$ we denote by $\str{A}^-$ its {\em relational reduct}
constructed as follows.
The language $L^-$ of $\str{A}^-$ is a relational language containing all relational symbols
of $\str{A}$ and additionally containing for every function symbol $\func{}{}\in L_\mathcal F$ a relation symbol
$\rel{}{F}\in L_\mathcal F^-$ of arity 2.
The vertex set of $\str{A}$ is the same as the vertex set of $\str{A}^-$.
For every $\rel{}{}\in L_\mathcal R$ we have $\rel{A}{}=\nbrel{\str{A}^-}{}$
and for every $\func{}{}\in L_\mathcal F$ it holds that  $(t_1,t_2)\in \relfunc{A}{}$ if and only if $t_2\in\func{A}{}(t_1)$.
\end{defn}

Note that in the above, the structures $\str{A}$ and $\str{A}^-$ have the same automorphisms and any partial automorphism of $\str{A}$ is a partial automorphism of $\str{A}^-$ (but of course not conversely). 

The proof of Theorem~\ref{EPPA} will occupy the rest of this section.

\begin{proof}[Proof of Theorem~\ref{EPPA}]

Given $\str{A}$ we invoke Theorem~\ref{thm:Herwig} to obtain a coherent EPPA-witness $\str{B}^-$  of $\str{A}^-$ (with respect to all partial isomorphisms of $\str{A}^-$). We use the following terminology, following an exposition by Hodkinson in~\cite{hodkinson,hodkinson2003}.
The difference between our proof and that in \cite{hodkinson2003} and \cite{Siniora} (which give clique faithful EPPA for structures in relational languages) is that whereas in these papers, the vertices in the EPPA-witness are of the form $(v, \chi_v)$ (as in the notation below), our vertices will actually be sets of such vertices (carrying an $L$-structure).

A set $S\subseteq B^-$ is called {\em small} if there is some $g\in \Aut(\str{B}^-)$ such that $g(S)\subseteq A$. Otherwise $S$ is called {\em big}.
Denote by $\mathcal U$ the set of all big subsets of $B^-$ and note that this is preserved by $\Aut(\str{B}^-)$. 
Given $b\in B^-$ a map $\chi:\mathcal U\to \mathbb N$ is a {\em $b$-valuation function} if  $\chi(S)=0$ for all $b\notin S\in \mathcal U$ and $1\leq \chi(S)<\lvert S\rvert$ otherwise.

Given vertices $a,b\in B^-$ and their valuation functions, $\chi_a$ and $\chi_b$ we say
that pairs $(a,\chi_a)$ and $(b,\chi_b)$ are {\em generic} if either $(a,\chi_a)=(b,\chi_b)$, or $a\neq b$ and for every $S\in \mathcal U$ such that $a,b\in S$ it holds that $\chi_a(S)\neq \chi_b(S)$.

The key construction in the proof is the following ``local covering''
construction of a \textit{$b$-valuation $L$-structure} $\str{V}_b$.  Fix $b\in B^-$ and
an automorphism $\alpha:B^-\to B^-$ such that $\alpha(b)\in \str{A}^-$ (we can
assume such an automorphism always exists --- all other vertices can be removed
from $\str{B}^-$). Now consider the $L$-substructure $\str{V}^\alpha_b=\cl_\str{A}(\alpha(b))$ of $\str{A}$.
Suppose that for every $v\in \alpha^{-1}(V^\alpha_b)$ we have a $v$-valuation function $\chi_v$
such that the assigned valuation functions are generic for every pair of vertices in $\alpha^{-1}(V^\alpha_b)$. (Such a choice of valuation functions always exists and we will show how to obtain it later when we define an embedding $\phi$ of $\str{A}$.)
Denote by $V_b$ the set of all such pairs $(v,\chi_v)$, $v\in \alpha^{-1}(V^\alpha_b)$.
On the set $V_b$ we consider the $L$-structure $\str{V}_b$, called a {\em $b$-valuation},
which is defined in such a way that the composition of mappings $\alpha$ and $\pi(v,\chi_v)=v$
forms an embedding $\alpha\circ\pi:\str{V}_b\to \str{V}^\alpha_b$. (This is a standard construction, we use the 1--1 mapping $\alpha\circ\pi$ to pull back the structure $\str{V}^\alpha_b$ to $V_b$; note that the
 isomorphism type of structure here does not depend on the choice of $\alpha$.)
Observe that then $\cl_{\str{V}_b}((b,\chi_b))=\str{V}_b$. 

Notice that for every $b$ there are multiple choices of $b$-valuations $\str{V}_b$ (depending on particular choice of valuation functions assigned to vertices, but not depending on the choice of $\alpha$). 
The sets $V_b$ and structures $\str{V}_b$ will form a ``cover of $\str{B}^-$'' and we find it convenient to make the following definitions.

\begin{defn}
Recalling that all functions of $L$ are unary, we say that a pair of valuations $\str{V}_a$ and $\str{V}_b$ is {\em generic} if
\begin{enumerate}
 \item[(i)] every pair of vertices $(u,\chi_u)\in V_a$ and $(v,\chi_v)\in V_b$ is generic;
 \item[(ii)] for every $(u,\chi_u)\in V_a$ and $(v,\chi_v)\in V_b$ and $\func{}{}\in L_\mathcal F$ it holds that
 \begin{enumerate}
   \item if $(u,v)\in \relfunc{B}{}$, then $(v,\chi_v)\in V_a$ and
   \item if $(v,u)\in \relfunc{B}{}$, then  $(u,\chi_u)\in V_b$;
 \end{enumerate}
 \item[(iii)] if $(u,\chi_u)\in V_a\cap V_b$, then $\cl_{\str{V}_a}((u,\chi_u)) = \cl_{\str{V}_b} ((u,\chi_u))$.
\end{enumerate}
We also say that a set $S$ of valuations is {\em generic} if every pair of valuations in $S$ is generic.  
\end{defn}

Now we construct an $L$-structure $\str{C}$: 
\begin{enumerate}
\item The vertices of $\str{C}$ are all $b$-valuation $L$-structures $\str{V}_b$, for $b\in B^-$.
\item For every relation $\rel{}{}\in L_\mathcal R$ put $(\str{V}_{v_1},\str{V}_{v_2},\ldots, \str{V}_{v_{\arity{}{}}})\in \rel{C}{}$ if and only if $(v_1,v_2,\ldots, v_{\arity{}{}})\in \nbrel{\str{B}^-}{}$ and the set $\{\str{V}_{v_i}:1\leq i\leq \arity{}{}\}$ is generic.
\item For every function $\func{}{}\in L_\mathcal F$ put $\func{C}{}(\str{V}_{v_1})=\{\str{V}_{v_2}, \str{V}_{v_3}, \ldots,\allowbreak \str{V}_{v_s}\}$ for some $s>1$ if and only if
\begin{enumerate}
 \item $\str{V}_2,\ldots,\str{V}_{s+1}$ are substructures of $\str{V}_1$,
 \item $\nbfunc{\str{V}_{v_1}}{}((v_1,\chi_1))=\{(v_2,\chi_2),(v_3,\chi_3),\ldots, (v_s,\chi_s)\}$ where 
 for every $1\leq i\leq s+1$ we denote by $\chi_i$ is the unique $v_i$-valuation such that $(v_i,\chi_i)\in V_{v_i}$,
 \item $({v_1},{v_l})\in \relfunc{B}{}$ for every $2\leq l\leq s+1$.
 \item the set $\{\str{V}_{v_i}:1\leq i\leq s+1\}$ is generic.
\end{enumerate}  
Put $\func{C}{}(\str{V}_{v_1})=\emptyset$ if the conditions above are not satisfied.
\end{enumerate}
We give an embedding $\phi:\str{A}\to\str{C}$ with generic image.
For every big set $S\in \mathcal U$
choose $f_S:S\to \{0,1,2,\ldots, \vert S\vert - 1\}$ to be a function such that $f_S(v)>0$ if and only if $v\in A\cap S$ and 
for every pair of vertices $u,v\in A\cap S$ it holds that $f_S(u)\neq f_S(v)$. Such a function exists because $A\cap S$ is always a proper subset of $S$. Given a vertex $a\in A$ we put $\phi(a)$
to be an $a$-valuation constructed from $\cl_\str{A}(a)$ by mapping every vertex $v\in \cl_\str{A}(a)$
to $(v,\chi_v)$ where $\chi_v(S)=f_S(v)$. It is easy to verify that this is indeed an embedding from $\str{A}$ to
$\str{C}$ and $\phi(\str{A})$ is generic.

We aim to show that $\str{C}$ is an EPPA-witness of $\phi(\str{A})$. We first take time to introduce a terminology and  prove a lemma which will allow  us to use the fact that $\str{B}^-$ is an EPPA-extension of $\str{A}^-$.

Denote by $\mathcal V$ the union of all vertex sets of $\str{V}_v\in \str{C}$.
If $g\in \Aut(\str{B}^-)$, we say that the partial map $q:\mathcal V\to \mathcal V$ is 
{\em $g$-compatible} if for all $(a,\chi)\in \dom(q)$ there exists a $g(a)$-valuation function $\chi'$
such that $q((a,\chi))=(g(a),\chi')$.

Similarly, let $g\in \Aut(\str{B}^-)$ and $p:\str{C}\to \str{C}$ be a partial automorphism.
 We say that $p$ is {\em $g$-compatible} if there exists a $g$-compatible map $q:\mathcal V\to \mathcal V$
such that for all $\str{V}_v\in \dom(p)$ $q$ restricted to $V_v$ is an isomorphism of $\str{V}_v$ and $p(\str{V}_v)$.

Denote by $\pi$ the homomorphism (called a \emph{projection}) $\str{C}^-\to \str{B}^-$ defined by $\pi(\str{V}_v)=v$.

\begin{lem}
\label{lem:technical}
 Let $p:\str{C}\to \str{C}$ be a partial automorphism with generic domain and range,
$g\in \Aut(\str{B}^-)$, and suppose that $p$ is $g$-compatible. Then $p$ extends to some $g$-compatible $\hat{p}\in \Aut(\str{C})$.
\end{lem}
\begin{proof}
As $\dom(p)$ is generic, for every $v\in \pi(\dom(p))$ there is precisely one $v$-valuation function $\chi_v$ such that
the pair $(v,\chi_v)$ is a vertex of some valuation $V_b\in \dom(p)$. Denote by $D$ the set of all such pairs (so $D = \bigcup \dom(p)$).
The same is true for the range and denote by $R$ all pairs appearing as vertices in valuations in $p(\dom(p))$.
It follows that $p$ uniquely defines a $g$-compatible map $q:D\to R$.
Fix a big set $S\in \mathcal U$. Then the set of pairs
$$\left\{\left(\chi_b(S),\chi'_{g(b)}\left(g\left(S\right)\right)\right):\left(b,\chi_b\right)\in D, \, q(b, \chi_b) = (g(b), \chi'_{g(b)})\right\}$$
is the graph of a partial permutation of $\{0,1,2,3,\ldots, \lvert S\rvert-1\}$ fixing 0 if defined on it. Extend it to a permutation $\theta_S^p$ of  $\{0,1,2,3,\ldots, \lvert S\rvert -1\}$ fixing 0.

Now we define $\hat q:\mathcal V\to \mathcal V$ by mapping $(b,\chi)\in \mathcal V$ to $(g(b),\chi')$ such that $\chi'(g(S)))=\theta_S^p(\chi(S))$
and $\hat p:\str{V}_v\to \str{V}_{g(v)}$ where $\str{V}_{g(v)}$ is created from $\str{V}_v$ by mapping every vertex $(b,\chi)\in V_v$ to $\hat q(b,\chi)$.

It is easy to verify that $\hat q$ is a well defined permutation of $\mathcal V$ which extends $q$ and is $g$-compatible. Moreover it preserves the relation of genericity between elements of $\mathcal{V}$.  Therefore also $\hat p$ is a well-defined permutation of $C$ which preserves generic sets, extends $p$ and is $g$-compatible. Consequently $\hat p$  is an automorphism of $\str{C}$.
\end{proof}

By Lemma~\ref{lem:technical} the extension property of $\str{C}$ for partial isomorphisms of $\phi(\str{A})$ follows easily.
Let $p$ be a partial isomorphism of $\phi(\str{A})$. We extend it to $\hat p\in \Aut(\str{C})$.
First extend $\phi^{-1}\circ p$ (which is an partial automorphism
of $\str{A}$) to automorphism $g\in \Aut(\str{B}^-)$. Clearly $p$ is $g$-compatible and because domain and range are generic. By Lemma~\ref{lem:technical}
$p$ extends to $g$-compatible $\hat p\in \Aut(\str{C})$. This shows that $\str{C}$ is indeed an extension of $\phi(\str{A})$.

\medskip

For coherence, we use a similar argument to that in~\cite{Siniora2, Siniora}. Given a coherent triple $(f_0,g_0,h_0)$  of partial automorphisms of $\str{A}$ we first extend this to a coherent triple $(f,g,h)$ of automorphisms of $\str{B}^-$. We let $(f_1,g_1, h_1)$ be the coherent triple of partial automorphisms of $\phi(\str{A})$ induced by $(f_0, g_0, h_0)$. Using Lemma~\ref{lem:technical} we extend $f_1$ to an $f$-compatible $\hat{f} \in \Aut(\str{C})$. Similarly  we obtain extensions $\hat{g}$, $\hat{h}$ of $g,h$. In order to ensure that the triple $(\hat{f}, \hat{g},\hat{h})$ is coherent, we only need to ensure that, in the proof of Lemma~\ref{lem:technical}, the permutations $\theta_S^p$ can be chosen coherently. More precisely, we want to ensure that $\theta_S^g\theta_S^f = \theta_S^h$. As in \cite{Siniora}, if we extend any partial permutation $\alpha$ on $\{1,\ldots, s\}$ to a permutation by mapping $\{1,\ldots, s\}\setminus \dom(\alpha)$ to $\{1,\ldots, s\}\setminus \alpha(\dom(\alpha))$ in an order-preserving way, then we obtain the required coherence.

\medskip

Finally we verify that  $\str{C}$ is faithful for irreducible
substructures.  
Let $\str{D}$ be an irreducible substructure of $\str{C}$. 
We first show that $\str{D}$ is generic.  Suppose not and that 
$\str{V}_a,\str{V}_b\in D$ form a  non-generic pair of vertices. Let $\str{E}_a = \{\str{V}_v \in D : \str{V}_a \not\in \cl_{\str{D}}(\str{V}_v)\}$. As closures are unary, this is a (proper) substructure of $\str{D}$. Similarly define $\str{E}_b$. Note that $\str{E}_a\cup \str{E}_b = \str{D}$: otherwise, there is $\str{V}_v \in \str{D}$ with $\str{V}_a, \str{V}_b \in \cl_{\str{D}}(\str{V}_v)$ and then $\str{V}_a, \str{V}_b \subseteq \str{V}_v$, so form a generic pair. Moreover, no relation of $\str{C}$ can involve a vertex $\str{V}_u \in \str{E}_a\setminus \str{E}_b$ and a vertex $\str{V}_v \in \str{E}_b\setminus \str{E}_a$ as $\str{V}_b \subseteq \str{V}_u$ and $\str{V}_a \subseteq \str{V}_v$, which implies that $\str{V}_u$, $\str{V}_v$ is not a generic pair. Thus $\str{D}$ is a free amalgam of the substructures $\str{E}_a$ and $\str{E}_b$, which is a contradiction to its irreducibility. So $\str{D}$ is generic.

Because $\str{D}$ is generic it follows that $S = \pi(D)$ is small. Indeed, for each $u \in S$, there is a $u$-valuation $\chi_u$ such that the set of pairs $\{(u,\chi_u): u \in S\}$ is generic. If $S$ were big, this would imply that $\{\chi_u(S) : u \in S\}$ has size $\vert S \vert$, which is impossible (its elements $j$ satisfy $1 \leq j < \vert S \vert$).

It follows that there is $g\in \Aut(\str{B}^-)$ such that $g(\pi(\str{D}))\subseteq A$.
The map $p : \str{D} \to \phi(\str{A})$ given by $p(\str{V}_u) =  \phi(g(u))$ is a $g$-compatible partial automorphism of $\str{C}$ with  generic domain and range. By Lemma~\ref{lem:technical}, $p$ extends to $\hat p\in \Aut(\str{C})$ and $\hat p(\str{D})\subseteq \phi(\str{A})$. This completes the proof that $\str{C}$ is faithful for irreducible substructures.
\end{proof}
\begin{remark}
The construction above adds an extra tool to the existing constructions of
EPPA-witness and can be thus used as an additional layer in the construction of
EPPA-witness for non-free amalgamation classes based on application of
Herwig-Lascar theorem~\cite{herwig2000,solecki2009,otto2014}. An example of such application
is given in~\cite{Aranda2017} giving EPPA for some classes of antipodal metric spaces.
\end{remark}

\section{Applications}
\label{sec:examples}

In this section we discuss how some previously-studied classes of structures can be viewed naturally as free amalgamation classes of structures with set-valued functions.

Before doing this, we mention an alternative viewpoint for classes of structures with closures which we used in~\cite{Evans2}.  In the following examples we will show how this is related to our definitions.

Consider a class $\K$ of finite $L$-structures, closed under isomorphisms,  and a distinguished class $\sqsubseteq$ of embeddings between elements of $\K$, called \emph{strong embeddings}. We shall assume $\sqsubseteq$ is closed under composition and contains all isomorphisms. 
In this case, we refer to $(\K; \sqsubseteq)$ as a \emph{strong class}. If $\str{A}$ is a substructure of $\str{B} \in \K$ and the inclusion map $\str{A} \to \str{B}$ is in $\sqsubseteq$, then we say that $\str{A}$ is a \textit{strong substructure} of $\str{B}$ and write $\str{A} \sqsubseteq \str{B}$. In the other words, a strong class is a subcategory of $\K$ with the strong embeddings. 

The Ramsey property and amalgamation property can be defined analogously to the Ramsey property and amalgamation property
of classes of $L$-structures, but considering only strong substructures and strong embedding.
Most of the \Fraisse{} theory remains valid in this setting (see~\cite{Evans2} for more details).

\subsection{$k$-orientations}
For a fixed natural number $k$, a {\em $k$-orientation} is an oriented (that is, directed) graph such that the out-degree of
every vertex is at most $k$.  We say that a substructure $G_1=(V_1,E_1)$  of a 
$k$-orientation  $G_2=(V_2,E_2)$ is {\em successor closed}  if
 there is no edge from $V_1$ to $V_2\setminus V_1$ in $G_2$. 

Denote by $\mathcal D_k$ the class of all finite $k$-orientations.  This is
a hereditary class closed for free amalgamation over successor-closed subgraphs
and thus the successor-closedness plays the r\^ole of strong substructure, so $\mathcal D_k$ can be considered as a class with corresponding strong embeddings.
We show how to turn $\mathcal D_k$ into a free amalgamation class in the sense of Definition~\ref{defn:amalg}.

Given an oriented graph $G=(V,E)\in \mathcal D_k$ denote by $\str{G}^+$ the structure with vertex set
$V$ and (partial) unary function $\func{}{}$. The function $\func{}{}$
maps every vertex to its out-neighborhood.
Denote by $\mathcal D^+_k$ the class of all structures $\str{G}^+$ for $G\in \mathcal D_k$.
Because $\str{G}^+_1$ is a substructure of $\str{G}^+_2$ if and only if $G_1$ is successor
closed in $G_2$ it follows that $\mathcal D^+_k$ is a free amalgamation class.
We immediately obtain:
\begin{thm}
\label{orientations}
 The class
$\vv{\mathcal D}^+_k$ is Ramsey and there exists a 
class $\mathcal O_k\subseteq \vv{\mathcal D}^+_k$ with the ordering property (with respect to $\mathcal D^+_k$).
The class $\mathcal D^+_k$ has the extension property for partial automorphisms.
\end{thm}
Let us briefly discuss what is the structure of $\mathcal O_k$. Given
$\ostr{A}\in \mathcal O_k$, the closure-components (recall
Definition~\ref{def:closure-components}) of $\ostr{A}$ corresponds to strongly
connected components in the underlying oriented graph and $\ostr{A}$ is an
ordered closure-extension of level $0$ if and only if the underlying graph is
strongly connected. More generally $\ostr{A}$ is an
ordered closure-extension of level $k$ if it contains a single strongly connected component
$C$ of level $k$ and all other vertices of $\ostr{A}$ are reachable from
$C$ via an oriented path. Condition~\ref{o:unique} of
Definition~\ref{def:admissible} thus requires that the ordering of $C$ is
determined by the isomorphism type of $\str{A}$ (the underlying oriented graph)
and the ordering of $A\setminus C$. Thus in $\mathcal O_k$, vertices are ordered primarily
by the number of vertices in their closure. Every closure-component forms an interval where the order within this interval is fixed by the similarity type of corresponding closure-extension.
The relative order of closure-components is given by their isomorphism type and the 
ordering of closure-components reachable from them. This can be seen as a generalization
of the order of oriented forests described in Section~\ref{sec:ordering}.

Theorem~\ref{orientations} can be seen as the most elementary use of Theorems~\ref{thm:main}, \ref{thm:ordering}, and \ref{EPPA},
but it has important consequences.
Denote by $\mathcal C_k$ the undirected reducts of oriented graphs in $\mathcal D_k$, that is,  the class
of all unoriented graphs which can be oriented to an $k$-orientation.
Given a graph $G=(V,E)$, its \emph{predimension} is $\delta(G) = k \lvert V \rvert - \lvert E\rvert$.
It is the heart of Hrushovski predimension construction that the class $\mathcal C_k$ forms a free amalgamation class
for the following notion of strong subgraph.
Given a graph $G\in \mathcal C_k$ its subgraph $H$ is a {\em self sufficient} or {\em strong subgraph} if
for every subgraph $H'$ of $G$ containing $H$ it holds that $\delta(H)\leq \delta(H')$.

The connection between the Hrushovski predimension construction and orientability follows by the Marriage Theorem
and was first introduced in~\cite{Evans2003, Evans2005}.  Its consequences in Ramsey theory are the main topic of~\cite{Evans2}
and they are out of scope of this paper. We however point out why this free amalgamation class over strong subgraphs
does not translate to a free amalgamation class when enriched by set-valued functions representing the smallest
self-sufficient subgraph of a given set. Consider a graph in $\mathcal C_2$ created as amalgamation
depicted in Figure~\ref{fig:nonfree}.
\begin{figure}[t]
\centerline{\includegraphics{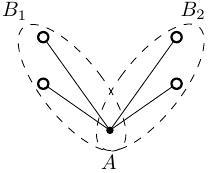}}
\caption{An amalgamation of $B_1,B_2\in \mathcal C_2$ over $A$.}
\label{fig:nonfree}
\end{figure}
While in both $B_1$ and $B_2$ the vertices denoted by circles forms a self-sufficient substructures,
it is not the case in the free amalgamation. The predimension of the 4 independent vertices is 8, while the
predimension of the whole amalgam is 6.
 It follows that in order to represent self-sufficient substructures
by means of set-valued functions, a new function from the vertices denoted by circles would need to be added.
This makes the amalgamation non-free in our representation and this is the reason why additional information about orientation of the edges
is needed.
\subsection{Steiner systems}
\label{sec:steiner}
It was was established in~\cite{bhat2016ramsey,Hubicka2016,Hubicka2017graham} that the class of finite partial Steiner systems is Ramsey with respect to strong
subsystems. Moreover, the  ordering property follows from techniques of~\cite{Nesetril1975}. We derive both results 
by a re-interpretation of partial Steiner systems as a free amalgamation class in a functional language.

This is an example where non-unary functions are necessary.

\medskip

\begin{figure}[t]
\centerline{\includegraphics{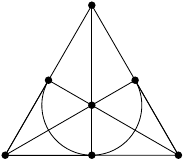}}
\caption{Fano plane (Steiner (2,3)-system).}
\label{fig:fano}
\end{figure}
For fixed integers $r\geq t\geq 2$, by a  {\em partial Steiner $(r,t)$-system} we mean an $r$-uniform
hypergraph $G=(V,E)$ with the property that every $t$-element subset of $V$ is
contained in at most one edge of $G$. (If there is always exactly one such edge, we have a Steiner system.) Abusing terminology somewhat, we shall refer to this simply as a Steiner $(r,t)$-system.
Given two Steiner $(r,t)$-systems $G$ and $H$, we say that $G$ is a {\em strongly
induced subsytem} of $H$ if
\begin{enumerate}
 \item $G$ is an induced subhypergraph of $H$; and,
 \item every hyperedge of $H$ which is not a hyperedge of $G$ intersects 
  $G$ in at most $t-1$ vertices.
\end{enumerate}

In~\cite{bhat2016ramsey} the Ramsey property was formulated with respect to strongly
induced subsystems.  We, equivalently, use set-valued functions to represent this.

\begin{defn}
Denote by $\mathcal S_{r,t}$ the class of all finite structures $\str{A}$ with one 
function $\func{}{}$ from $t$-tuples with the following properties:
\begin{enumerate}
  \item If $\func{A}{}(\vec{x})\neq \emptyset$ then $\vec{x}$ has no repeated vertices, $|\func{A}{}(\vec{x})|=r$ and every vertex of $\vec{x}$ is in $\func{A}{}(\vec{x})$.
  \item Whenever $\vec{x}$ is an $t$-tuple of vertices of $\str{A}$, $\vec{x}_2$ is an $t$-tuple consisting of distinct vertices in $\func{A}{}(\vec{x})$ it holds that
 $\func{A}{}({\vec{x}})=\func{A}{}(\vec{x}_2)$.
\end{enumerate}
\end{defn}
It is easy to see that $\mathcal S_{r,t}$ is a free amalgamation class.

Given a Steiner $(r,t)$-system $G=(V,E)$ we can interpret it as a structure
$\str{S}_G\in \mathcal S_{r,t}$ with vertex set $V$ and function $\func{}{}(\vec{x})$ defined for every $t$-tuple $\vec{x}$ of distinct vertices such that there is hyperedge $A\in E$
containing all vertices of $\vec{x}$. In this case we put $\func{}{}{}(\vec{x})=A$.

Observe that if $G$ is a strong subsystem of $H$ if and only if $\str{S}_G$
is a substructure of $\str{S}_H$. It follows that Steiner $(r,t)$-systems
are in 1--1 correspondence to structures $\str{S}_G$ and moreover this correspondence
maps subsystems to substructures.

We obtain an alternative proof of the following main result of~\cite{bhat2016ramsey}:
\begin{thm}
The class $\vv{\mathcal S}_{r,t}$ is a Ramsey class with the ordering property.
\end{thm}
\begin{proof}
The Ramsey property follows directly from Theorem~\ref{thm:main}. For the ordering property, note that single vertices are closed in structures in $\mathcal{S}_{r,t}$, so all orderings are admissible, in the sense of Theorem ~\ref{thm:ordering}. 
\end{proof}
In \cite{Hubicka2017graham} we obtained further corollaries to this approach.
We remark that EPPA is presently open for the class of partial Steiner systems. Of course, our results in Section~\ref{sec:EPPA} does not apply in this case, as the function introduced is not unary.

\subsection{Bowtie-free graphs}
\begin{figure}[t]
\centerline{\includegraphics{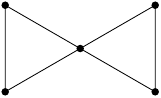}}
\caption{The bowtie graph.}
\label{fig:bowtie}
\end{figure}
A {\em bowtie} $B$ is a graph consisting of two triangles with one vertex
identified (see Figure~\ref{fig:bowtie}). A graph $G$ is {\em bowtie-free}
if there is no monomorphism from $B$ to $G$.  The existence of a 
universal graph in the class of all countable bowtie-free graphs was shown in
\cite{Komjath1999}. The paper \cite{Cherlin1999} gave a far reaching generalization
by giving a condition for the existence of $\omega$-categorical universal graphs for classes defined by forbidden
monomorphisms (which we refer to as to as {\em Cherlin-Shelah-Shi classes}). This  led to several new classes being identified~\cite{Cherlin2007, Cherlin2015, Cherlinb}. Bowtie-free graphs represent a key example of a class that
is not a free amalgamation class by itself, but can be turned into one by means
of unary functions.  This analysis was carried in~\cite{Hubivcka2014}  where we gave an explicit characterisation of the ultrahomogeneous lift and the Ramseyness of this lift.
The presentation can be greatly simplified by considering
structures with set-valued functions and moreover we show the extension property
for partial automorphisms (See also~\cite{Siniora2} for related results on ample generics). 

 While not all Cherlin-Shelah-Shi classes give rise to free
amalgamation classes (see the more detailed analysis in~\cite{Hubicka2016}),
what follows can be generalized to many of the other block-path examples given
by~\cite{Cherlin2015} and \cite{Cherlinb}. 

We review the main observations about the structure of bowtie-free graphs from~\cite{Hubivcka2014}. For completeness we include the (easy) proofs.

\begin{defn} [Chimneys]
For $n\geq 2$, an {\em $n$-chimney graph}, $Ch_n$, is a free amalgamation of $n$ triangles over one common edge.
A {\em chimney graph} is any graph $Ch_n$ for some $n\geq 2$.
\end{defn}

Chimneys together with $K_4$ (a clique on 4 vertices) will form the only components of bowtie-free graphs formed by triangles. The assumption $n \geq 2$  for chimney is a technical assumption to avoid isolated triangles. Note also that $Ch_2$ is not an induced subgraph of $K_4$.
 
\begin{defn}[Good bowtie-free graphs]
\label{def:goodgraph}
A bowtie-free graph $G=(V,E)$ is {\em good} if every vertex is contained either in a copy of chimney or a copy of the complete graph $K_4$.
\end{defn}

The structure of bowtie-free graphs is captured by means of the following three lemmas:

\begin{lem}[\cite{Hubivcka2014}]
\label{lem:bowtiestructure}
Every bowtie-free graph $G$ is an induced subgraph of some good bowtie-free graph $G'$.
\end{lem}
\begin{proof}
The graph $G$ can be extended in the following way:
\begin{enumerate}
 \item[1.] For every vertex $v$ not contained in a triangle add a new copy of $Ch_2$ and identify  the vertex $v$ with one of the vertices of $Ch_2$.
 \item[2.] For every triangle $v_1,v_2,v_3$ that is not part of a 2-chimney nor a $K_4$, add a new vertex $v_4$ and the triangle $v_1,v_2,v_4$
       turning the original triangle into $Ch_2$.
\end{enumerate}
It is easy to see that step $1$ cannot introduce a new bowtie.

Assume, to the contrary, that step $2$ introduced a new bowtie. Further assume that $v_1$ is the
unique vertex of degree 4 of this new bowtie and consequently there is another
triangle on vertex $v_1$ in $G$.  Because $G$ is bowtie-free, this triangle
must share a common edge with the triangle $v_1,v_2,v_3$ and therefore the 
triangle $v_1,v_2,v_3$ is already part of a $K_4$ or a 2-chimney in the original graph $G$.  A
contradiction.
\end{proof}

Now we are ready to describe how to turn the class of bowtie-free graphs into a free amalgamation class. Our language $L$
will consist of one binary relation $\rel{}{}$ and unary function $\func{}{}$.

\begin{defn}
\label{defn:Hplus}
For every good bowtie-free graph $G=(V,E)$ denote by $\str{G}^+$ the
$L$-structure with vertex set $V$ and relations and
functions defined as follows:
\begin{enumerate}
 \item $(u,v)\in \rel{G}{}$ if and only if $\{u,v\}$ is an edge of $G$.
 \item $\func{G}{}(v)=u$ if and only if $\{u,v\}$ is contained in at least two triangles of a chimney.
 \item $\func{G}{}(v)=\{u_1,u_2\}$ if and only if $\{v,u_1,u_2\}$ is a triangle of a chimney and $v$ is not contained in multiple triangles.
 \item $\func{G}{}(v)=\{u_1,u_2,u_3\}$ if and only if $\{v,u_1,u_2,u_3\}$ forms a 4-clique in $G$.
\end{enumerate}
\end{defn}

Denote by $\mathcal B$ the class of all $\str{A}$-partite substructures of structures $\str{G}^+$ where $G$ is a good bowtie-free graph.

\begin{thm}
$\mathcal B$ is a free amalgamation class.
\end{thm}
\begin{proof}
Let $\str{A}, \str{B},\str{B}'\in \mathcal B$.  Assume that $\str{A}$ is a substructure
of both $\str{B}$ and $\str{B}'$. We show that the free  amalgamation $\str{C}$ of $\str{B_1}$ and $\str{B}_2$
over $\str{A}$ is in $\mathcal B$.

There are good bowtie free graphs $G_1$ and
$G_2$ such that $\str{B}\subseteq\str{G}^+_1$ and $\str{B}' \subseteq\str{G}^+_2$. We claim that 
that the free amalgamation $H$ of $G_1$ and $G_2$ over $A$ is a good bowtie-free graph. Consequently $\str{H}^+$ (given by Definition~\ref{defn:Hplus}) is the free amalgam of $\str{G}_1^+$ and $\str{G}_2^+$ over $\str{A}$. As $\str{C}$ is a substructure of $\str{H^+}$, it then follows that $\str{C} \in \mathcal{B}$.

Because $\str{A}$ is a substructure of both $\str{G}_1^+$ and $\str{G}_2^+$, the function $\func{}{}$ ensure that the free amalgamation preserves the  structure of chimneys:
if a vertex of a chimney in $G_1$ is identified with a vertex of a chimney in $G_2$ (because it is in $A$) 
then also the bases (i.e. the edges in multiple triangles) of these chimneys are contained in $A$, so are identified in $H$ and the result is again a chimney.
Similarly $\func{}{}$ makes sure that a $4$-clique containing a vertex of $A$ is in $A$.  Finally
free amalgamation cannot introduce any new triangles and thus the free amalgamation
is a good bowtie-free graph $H$. Consequently $\str{H}^+$ is the free amalgam of $\str{G}_1^+$ and $\str{G}_2^+$ over $\str{A}$. 
\end{proof}

\begin{corollary} The class
$\mathcal B$
 has irreducible-structure faithful EPPA;
 $\vv{\mathcal B}$ is a Ramsey class and there is $\mathcal B'\subseteq \vv{\mathcal B}$
with the ordering property (with respect to $\mathcal{B}$). 
\end{corollary}
The class $\mathcal B'$ can be easily derived from the Definition~\ref{def:admissible}.
\begin{figure}[t]
\centering
\includegraphics{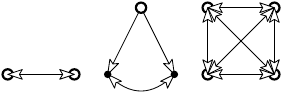}
\caption{All isomorphism types of closure-extensions in $\mathcal B$.}
\label{fig:closures}
\end{figure}
There are only three types of closure-extensions in $\mathcal B$ depicted in Figure~\ref{fig:closures} (with arrows representing function $\func{}{}$ and circles denoting the vertices of maximal level).
It follows that vertices are ordered by size of their closures (here we make use of the definition of $\preceq$ refining the order given by number of vertices).
That is vertices in bases of chimney are first, vertices in the top of chimneys next and vertices in 4-cliques last.
Every vertex-closure forms an interval. Pair of vertices of level 1 (in the top of chimneys) form
homologous extensions if and only if they belong to the same chimney. It follows that for every chimney
the set of its top vertices forms an interval and the  relative order of these intervals corresponds to the relative order of corresponding bases.

\begin{remark}
The Ramsey property and an explicit description of the admissible
ordering was given in~\cite{Hubivcka2014}.  The relational language used  there is however
more complicated and does not preserve all automorphisms of the \Fraisse{}
limit of $\mathcal B$. This makes it unsuitable for the extension property for partial automorphisms.
The formulation here is a more optimized version.

The EPPA for bowtie-free graphs was states as a problem in Siniora's thesis~\cite{Siniora2} and is attributed to Macpherson.

The argument above together with the observation that in the \Fraisse{} limit $\str{B}$ of $\mathcal B$ 
we have that for every finite $S\subseteq B$, $\vert \cl_\str{B}(S)\vert \leq 3\lvert S \rvert$, also gives
a compact proof for the existence of an $\omega$-categorical countable universal bowtie-free graph.
This bound follows from the fact that function $\func{}{}$ cannot cascade.
The $\omega$-categoricity follows from the fact that the orbit of $S$ in $\Aut(\str{B})$ is fully determined by the isomorphism type of $\cl_\str{B}(S)$
and there are only finitely many closures for every finite $S$.
This, of course, is just a re-formulation of the argument in~\cite{Cherlin1999}.
\end{remark}
\section{Concluding remarks}
\paragraph{1.}
It would be interesting to extend Theorem~\ref{EPPA} to a class of structures which include non-unary functions. Perhaps this is too much to ask as EPPA is presently open even in the case of Steiner triple systems (as we remark in Section~\ref{sec:steiner}). However note that our structures involve set-valued functions and thus the EPPA may be easier to prove. However even for partial triple systems the EPPA seems to be presently open.
\paragraph{2.} On the structural Ramsey theory side open problems include Ramsey properties of finite lattices and other algebraic structures where the axioms (such as associativity) are difficult to control in an amalgamation procedure. See \cite{Hubicka2016,Aranda2017} for results on Ramsey classes.
\bigskip

\noindent\textit{Acknowledgements:\/ }
We would like to thank to Daoud Siniora for several useful discussions concerning
clique faithful EPPA and the notion of coherency.
We are also grateful to the anonymous referee or detailed report which improved
presentation of this paper.

\bibliography{ramsey.bib}
\end{document}